\documentclass[reqno]{amsart}
\usepackage[dvips]{graphicx}

\usepackage{amssymb}

\newcommand{\N}{\mathbb{N}}
\newcommand{\Z}{\mathbb{Z}}

\newcommand{\R}{\mathbb{R}}

\newtheorem{thm}{Theorem}[section]
\newtheorem{prop}[thm]{Proposition}
\newtheorem{assume}[thm]{Assumption}

\newtheorem{lem}[thm]{Lemma}
\newtheorem{rem}[thm]{Remark}
\newtheorem{cor}[thm]{Corollary}
\newtheorem{conj}[thm]{Conjecture}
\newtheorem{ex}[thm]{Example}

\begin{document}
\title[Accumulation points of Coxeter groups]
{Distribution of accumulation points of roots for type $(n-1,1)$ Coxeter groups}
\author[A. Higashitani]{Akihiro Higashitani}
\author[R. Mineyama]{Ryosuke Mineyama}
\author[N. Nakashima]{Norihiro Nakashima}
\thanks{
\hspace{-0.5cm}{\bf 2010 Mathematics Subject Classification:} 
Primary 20F55, 51F15; Secondary 05E15. 
The first author is supported by JSPS Research Fellowship for Young Scientists. }
\address{Akihiro Higashitani,
Department of Mathematics, Kyoto Sangyo University, 
Kamigamo Motoyama, Kita-ku, Kyoto 603-8555, Japan}
\email{ahigashi@cc.kyoto-su.ac.jp}
\address{Ryosuke Mineyama, 
Department of Mathematics, 
Graduate School of Science,
Osaka University,
Toyonaka, Osaka 560-0043, Japan}
\email{r-mineyama@cr.math.sci.osaka-u.ac.jp}
\address{Norihiro Nakashima, 
Department of Mathematics, Tokyo Denki University, Tokyo 120-8551, Japan}
\email{nakashima@mail.dendai.ac.jp}

\begin{abstract}
In this paper, we investigate the set of accumulation points of 
normalized roots of infinite Coxeter groups for certain class of their action. 
Concretely, we prove the conjecture proposed in \cite[Section 3.2]{hlr} 
in the case where the equipped Coxeter matrices are of type $(n-1,1)$, where $n$ is the rank. 
Moreover, we obtain that the set of such accumulation points 
coincides with the closure of the orbit of one point of normalized limit roots. 
In addition, in order to prove our main results, 
we also investigate some properties on fixed points of the action. 
\end{abstract}

\maketitle

\section{Introduction}

The theory of Coxeter groups has been developed from 
not only combinatorial but also geometrical aspects. 
One of the most fundamental and important objects associated with Coxeter groups 
is root systems. 
In the case of a finite Coxeter group, which is nothing but a finite reflection group, 
its roots correspond to normal vectors of hyperplanes defining Euclidean reflections. 
In the case of an infinite Coxeter group, if it is an affine reflection group, 
which is a small class of infinite Coxeter groups, 
then its roots also correspond to normal vectors. 
However, little investigation on roots has been done for the case of general infinite Coxeter groups. 
This paper is devoted to analyzing roots of infinite Coxeter groups 
whose associated bilinear forms have the signature $(n-1,1)$. 
Concretely, we prove Conjecture \ref{yosou} below for all of such Coxeter groups. 

Hohlweg, Labb\'e and Ripoll proved that accumulation points of roots of infinite Coxeter groups 
lie in the projected isotropic cone $\widehat{Q}$ (\cite[Theorem 2.7]{hlr}). 
In addition, they conjectured in \cite[Section 3.2]{hlr} that 
the distribution of such points can be described as some appropriate set of points. 
From geometrical viewpoints, 
as is well known in the theory of discrete groups of M\"obius transformation, 
to study accumulation points is nothing but to study the interaction between ergodic theory and discrete groups.
That has rich geometrical aspects and the theory stands as a well developed branch of mathematical researches. 
In order to establish that theory, the hyperbolicity of the space plays a crucial role.
For the case where the associated matrices have signature $(n-1,1)$, 
Coxeter groups also have some hyperbolicity. 
This leads us to inspect an analogue of the theory of 
Kleinian groups for Coxeter groups of such class.

Recall that $W$ is a {\em Coxeter group} of rank $n$ with the generating set $S$ 
if $W$ is generated by the set $S=\{s_1, \ldots, s_n\}$ subject only to the relations 
$(s_is_j)^{m_{ij}}=1$, where $m_{ij} \in \Z_{>1} \cup \{\infty\}$ for $1 \leq i < j \leq n$ 
and $m_{ii}=1$ for $1 \leq i \leq n$. Thus, $m_{ij}=m_{ji}$. 
We say that the pair $(W,S)$ is a {\em Coxeter system}. 
We refer the reader to \cite{Hum90} for the introduction to Coxeter groups. 

For a Coxeter system $(W,S)$ of rank $n$, 
let $V$ be a real vector space with its orthonormal basis $\Delta=\{\alpha_s : s \in S\}$. 
Note that by identifying $V$ with $\R^n$, we treat $V$ as a Euclidean space. 
We define a symmetric bilinear form on $V$ by setting 
\begin{align*}
B(\alpha_i, \alpha_j) \; 
\begin{cases}
=-\cos\left(\frac{\pi}{m_{ij}}\right) \;\;\;&\text{if } m_{ij}< \infty, \\
\leq -1 &\text{if }m_{ij}=\infty 
\end{cases}
\end{align*}
for $1 \leq i \leq j \leq n$, where $\alpha_{s_i}=\alpha_i$, 
and call the associated matrix $B$ the {\em Coxeter matrix}. 
Classically, $B(\alpha_i, \alpha_j)=-1$ if $m_{ij}=\infty$, 
but throughout this paper, we allow its value to be any real number less than or equal to $-1$. 
This definition is derived from \cite{hlr} and this is available in some situations. 
Given $\alpha \in V$ such that $B(\alpha, \alpha) \not=0$, $s_\alpha$ denotes 
the map $s_\alpha : V \to V$ by
$$s_\alpha(v)=v- 2 \frac{B(\alpha,v)}{B(\alpha, \alpha)}\alpha \;\;\;\text{for any } v \in V,$$ 
which is said to be a {\em $B$-reflection}. 
Then $\Delta$ satisfies that 
\begin{itemize}
\item[(i)] for all $\alpha, \beta \in \Delta$ with $\alpha \not= \beta$, one has 
$$B(\alpha, \beta) \in (-\infty, -1] \cup \left\{ - \cos\left(\frac{\pi}{k}\right) : k \in \Z_{>1} \right\};$$ 
\item[(ii)] for all $\alpha \in \Delta$, one has $B(\alpha, \alpha)=1$. 
\end{itemize}
Such a set $\Delta$ is called a {\em simple system} and its elements are {\em simple roots} of $W$.
The Coxeter group $W$ acts on $V$ as composition of $B$-reflections and 
its generating set $S$ is identified with $\{s_{\alpha} : \alpha \in \Delta\}$. 
The {\em root system} $\Phi$ of $W$ is defined to be the orbit of $\Delta$ 
under the action of $W$ and its elements are called {\em roots} of $W$. 
The pair $(\Phi, \Delta)$ is said to be a {\em based root system} in $(V,B)$. 
We mention that $\Delta$ in \cite[Definition 1.2]{hlr} is assumed to be positively independent, 
while we assume the linearly independence throughout this paper.

\bigskip

Our main interest is the distribution of accumulation points of roots of an infinite Coxeter group. 
In the case of a finite Coxeter group, its root system $\Phi$ is finite. 
When a Coxeter group is infinite, $\Phi$ is also infinite. 
Thus the classical tools developed in the Euclidean geometry are no longer usable. 

On the other hand, in a recent paper \cite{hlr}, some tools to deal with roots of infinite Coxeter groups 
were established as the first step of their study. 
Our motivation to organize this paper is to contribute further studies of the paper \cite{hlr}. 

As is known in \cite[Theorem 2.7 (i)]{hlr}, 
the norm of a positive root always diverges as its depth tends to infinity. 
Thus, in order to investigate asymptotical behaviors of positive roots, 
it is needed to normalize them in the sense of a function $| \cdot |_1$, 
which will be defined in Section \ref{junbi}. We also set an affine subspace $V_1=\{x \in V : |x|_1=1\}$. 

Let $$\widehat{Q}=\{x \in V_1 : B(x,x)=0\}$$ 
and let $E$ be a set of accumulation points of normalized roots $\widehat{\rho}$ for $\rho \in \Phi$, 
i.e., the set consisting of all the possible limits of injective sequences of normalized roots. 
Let $w \cdot x$ denote the normalized action on $V_1$ for $w \in W$ and $x \in V_1$. 
(See Section \ref{junbi}.) 
It was proved in \cite[Theorem 2.7]{hlr} that $E \subset \widehat{Q}$ and the following is proposed. 
\begin{conj}[{\cite[Conjecture 3.9]{hlr}}]\label{yosou}
We say that $\Delta_I \subset \Delta$ is generating if $\widehat{Q} \cap \text{{\em span}}(\Delta_I)$ 
is included in $\text{{\em conv}}(\widehat{\Delta}) \cap \text{{\em span}}(\Delta_I)$. 
Let $E_I \subset E$ be the set of accumulation points of normalized roots of the parabolic subgroup associated with $\Delta_I$. 
Then we have the following properties: 
\begin{itemize}
\item[(i)] if $\Delta_I$ is generating, 
then $E_I=\widehat{Q} \cap \text{{\em span}}(\Delta_I)$; 
\item[(ii)] the set $E$ is the topological closure of the fractal self-similar subset $F_0$ 
of $\widehat{Q}$ defined by 
$$F_0:=W \cdot \left( \bigcup_{\substack{\Delta_I \subset \Delta \\
\Delta_I \text{{\em  is generating}}}} \widehat{Q} \cap \text{{\em span}}(\Delta_I) \right).$$ 
\end{itemize}
\end{conj}

In this paper, we prove the following theorem. 
\begin{thm}\label{mein}
For an infinite Coxeter group of rank $n$ equipped with the signature $(n-1,1)$ bilinear form, we have the following: 
\begin{itemize}
\item[(a)] When $\widehat{Q} \subset \mathrm{conv}(\widehat{\Delta})$, we have $E=\widehat{Q}$. 
\item[(b)] When $\widehat{Q} \not\subset \mathrm{conv}(\widehat{\Delta})$, 
we have $E=\widehat{Q} \setminus \left(\bigcup_{i=1}^m W \cdot D_i \right)$, where 
$D_1, \ldots, D_m$ are connected components of $\widehat{Q}$ out of $\mathrm{conv}(\widehat{\Delta})$ 
with $1 \leq m \leq n$. 
\end{itemize}
\end{thm}

Moreover, we also prove the following theorem. 

\begin{thm}\label{kidou}
Fix $x \in E$. Then $$\overline{W \cdot x} = E.$$ 
\end{thm}

We remark that Theorem \ref{mein} (a) and Theorem \ref{kidou} imply Conjecture \ref{yosou} 
in the case of Coxeter matrices whose signatures are $(n-1,1)$. For more details, see Remark \ref{kiwotukeru}. 

\begin{rem}\label{chuui}{\em 
It is easy to calculate that each Coxeter matrix arising from a Coxeter group of rank 3 
is either positive type or has the signature $(2,1)$ (cf. \cite[Section 6.7]{Hum90}). 
However, for a general Coxeter group of rank $n$, there exists a bilinear form whose signature is 
neither positive type nor  $(n-1,1)$. See Example \ref{example}. 
}\end{rem}

\begin{rem}{\em 
In \cite{dhr}, while revising the previous version of this paper, 
Dyer, Hohlweg and Ripoll also proved Theorem \ref{mein} and Theorem \ref{kidou} 
by a different approach (\cite[Theorem 4.10 (a) and Theorem 3.1 (b)]{dhr}). 
In fact, their approach was accomplished by using a method of so-called {\em imaginary cones} 
and they do not assume the linear independence of $\Delta$. 
On the other hand, in this paper, some other aspects of infinite Coxeter groups 
(e.g. metric on $\widehat{Q}$) are investigated. 
}\end{rem}

A brief overview of this paper is as follows. 
First, we will prepare some lemmas and collect fundamental facts in Section \ref{junbi} 
for the proofs of the main theorems. 
Next, in Section \ref{onepoint}, we will prove Theorem \ref{kidou}. 
We also study the fixed poitns of the normalized action in Section \ref{koteiten}. 
Before proving Theorem \ref{mein} in general case, 
we will show Theorem \ref{mein} for the case of rank 3 in Section \ref{syoumei}. 
Finally, in Section \ref{higher}, we will prove Theorem \ref{mein} for the case of an arbitrary rank. 
The discussion of the fixed points of the normalized action is used in the proof of Theorem \ref{mein} for the case of rank 3.

\subsection*{Acknowledgements} 
The authors would like to be grateful to Yohei Komori and Hideki Miyachi 
for their helpful advices and instructive discussions on Theorem \ref{mein} 
and also thank Yuriko Umemoto for having useful seminars on Coxeter groups together, 
which are the origin of this work. 
The authors also would like to express a lot of thanks to Matthew Dyer, Christophe Hohlweg and Vivien Ripoll 
for giving several helpful comments and fruitful suggestions on this paper. 
The authors also would like to thank to the anonymous referee for a lot of helpful comments.


\section{The normalized action and a metric on $\widehat{Q}$}\label{junbi}
In this section, we prepare some notation and lemmas for proving Theorem \ref{mein} and \ref{kidou}. 
After defining $\widehat{Q}$, 
we collect some fundamental results on the normalized action on $\widehat{Q}$ 
and define a metric on $\widehat{Q}$.

\begin{assume}
Unless otherwise noted, 
we always assume that Coxeter groups are irreducible 
and their Coxeter matrices have the signature $(n-1,1)$, where $n$ is the rank. 
\end{assume}
By \cite[Proposition 2.14]{hlr}, 
if the based root system $(\Phi, \Delta)$ is reducible and we consider proper subsets 
$\Delta_I, \Delta_J \subsetneq \Delta$ such that $\Delta = \Delta_I \sqcup \Delta_J$ 
with $B(\alpha,\beta)=0$ for all $\alpha \in \Delta_I$ and $\beta \in \Delta_J$, 
then $E(\Phi)=E(\Phi_I) \sqcup E(\Phi_J)$. 
Hence we may restrict our study to the irreducible cases.

As the following example shows, 
there exists a Coxeter group whose Coxeter matrix does not have the signature $(n-1,1)$. 

\begin{ex}\label{example}{\em 
Let $W$ be a Coxeter system of rank 4 with $S=\{s_1,s_2,s_3,s_4\}$ 
and $\Delta=\{\alpha_1,\alpha_2,\alpha_3,\alpha_4\}$. Let 
\begin{align*}
&B(\alpha_1,\alpha_2)=-a, \; B(\alpha_2,\alpha_3)=-b, \; B(\alpha_3,\alpha_4)=-c, \\
&B(\alpha_1,\alpha_3)=B(\alpha_1,\alpha_4)=B(\alpha_2,\alpha_4)=0, 
\end{align*}
where $a, b, c \in \{ \cos\left(\frac{\pi}{k}\right) : k \in \Z_{>2}\} \cup [1, \infty)$. 
It then follows from an easy computation that the signature of $B$ is $(2,2)$ 
if and only if $B$ is not positive type and 
three positive real numbers $a,b,c$ satisfy $a^2 + b^2 + c^2 - a^2c^2 < 1$. 
(Consult, e.g., \cite{Hum90} for the classification of positive type.) 
For example, when $(a,b,c)=(2,\frac{1}{2},2)$, this condition is satisfied. 

Thus, in the case of rank 4, there exists an infinite Coxeter group 
whose associated bilinear form has its signature $(2,2)$, 
while each Coxeter group of rank 3 is either positive type or of type $(2,1)$. 
}\end{ex}

Let 
\begin{align*}
	V^+ := 
		\{\ v \in V : v=\sum_{i=1}^n v_i\alpha_i, v_i \geq 0\}\;\text{ and }\; 
	V^- := 
		\{\ v \in V : v=\sum_{i=1}^n v_i\alpha_i, v_i \leq 0\}. 
\end{align*}
It is known that 
based root system allows us to define positive roots $\Phi^+ := \Phi \cap V^+$,
and then $\Phi = \Phi^+ \sqcup (-\Phi^+)$
(see, for instance, \cite{BD10,Kra09}).
In other words, all the roots are contained in $V^+ \cup V^-$.

\subsection{The normalized action of $W$}

First of all, we define $\widehat{Q}$ and discuss the action of $W$ on it. Let $$Q=\{v \in V : B(v,v)=0\}.$$ 

We fix the vector $o \in V$ as follows. 
If $B$ is of positive type, then $o=\sum_{i=1}^n \alpha_i$. 
If $B$ has the signature $(n-1,1)$, then $o$ is the eigenvector corresponding to the negative eigenvalue of $B$ 
whose Euclidean norm equals to $1$. 

By the following lemma, even in the case where $B$ has the signature $(n-1,1)$, 
we may assume that all coordinates of $o$ are positive. 
\begin{lem}\label{heimen}
	Let $o$ be an eigenvector for the negative eigenvalue of $B$.
	Then all coordinates of $o$ with respect to a basis $\Delta$ have the same sign.
\end{lem}
\begin{proof}
	This follows from Perron-Frobenius theorem for irreducible nonnegative matrices. 
	Let $I$ be the identity matrix of rank $n$. 
	Then $I-B$ is irreducible and nonnegative.
	Note that since $I-B$ and $B$ are symmetric, all eigenvalues are real.
	By Perron-Frobenius theorem, we have a positive eigenvalue $\lambda'$ of $I-B$
	such that $\lambda'$ is the maximum of eigenvalues of $I-B$ and 
	each entry of the corresponding eigenvector $u$ is positive.
	On the other hand, for each eigenvalue $a$ of $B$, there exists an eigenvalue $b$ of $I-B$
	such that $a = 1-b$.
	Let $\lambda$ be the negative eigenvalue of $B$.
	It then follows from an easy calculation that $\lambda = 1-\lambda'$. 
	Hence, $\R u = \R o$. Therefore, the positivity of each entry of $u$ implies that 
        the entries of $o$ are all positive or all negative. 
\end{proof}

Hence if we write $o$ for a linear combination $o = \sum_{i=1}^n o_i \alpha_i$ of $\Delta$, then $o_i > 0$ for each $i$.
Given $v \in V$, we define $$|v|_1=\sum_{i=1}^no_iv_i$$ if $v = \sum_{i=1}^n v_i \alpha_i$. 
Note that $|v|_1$ is nothing but the Euclidean inner product of $v$ and $o$. 
It is obvious that $|v|_1>0$ for $v \in V^+ \setminus \{0\}$ and $|v|_1<0$ for $v \in V^- \setminus \{0\}$. 
In particular, $|\alpha|_1 > 0$ for $\alpha \in \Delta$. 
Let $V_i=\{ v \in V\ \vert\ |v|_1=i\}$, where $i=0,1$. For $v \in V \setminus V_0$, 
we write $\widehat{v}$ for the ``normalized'' vector $\frac{v}{|v|_1} \in V_1$. 
Also for a set $A \subset V \setminus V_0$, 
we write $\widehat{A}$ for the set of all $\widehat{a}$ with $a \in A$.
We notice that since $B(x,\alpha) = |\alpha|_1B(x,\widehat{\alpha})$ holds, 
the sign of $B(x,\widehat{\alpha})$ is equal to the sign of $B(x,\alpha)$ for any $x \in V$ and $\alpha \in \Delta$.

As noted before, $W$ acts on $V$ as composition of $B$-reflections. 
For analyzing asymptotical aspects of $W$, we consider another action of $W$ on $V$. 
The {\em normalized action} of $w \in W$ on $V_1 \setminus W(V_0)$ is given by 
$$
	w \cdot v := \widehat{w(v)},\ \  v \in V,
$$
where $w(v)$ denotes the action of $w$ defined before 
(composition of $B$-reflections). 
This action is well-defined on $V_1 \setminus W(V_0)$.

\begin{lem}\label{zerodake}
We have $$W(V_0) \cap Q = \{{\bf 0}\},$$ where {\bf 0} is the origin of $\R^n$. 
\end{lem}
\begin{proof}
Since $Q$ is $W$-invariant, it is enough to show that $V_0 \cap Q= \{{\bf 0}\}$. 

For $i = 1,\ldots,n-1$, 
let $p_i$ be an eigenvector of $B$ with Euclidean norm 1 corresponding to each positive eigenvalue $\lambda_i$, respectively. 
Then, for any $v \in V_0$, we can express $v$ by a linear combination 
$v = \sum_{i=1}^{n-1} v_i p_i + v_n o$ for some $v_1,\ldots,v_n \in \R$ with respect to an orthonormal basis $p_1,\ldots,p_{n-1},o$ for $V$. 
Since $|v|_1=0$, we have $v_n=0$. Moreover, since $B$ is positive-definite 
in the subspace of $V$ spanned by $p_1,\ldots,p_{n-1}$, we have $B(v,v) = \sum_{i=1}^{n-1} \lambda_i v_i^2 \ge 0$. 
The positivity of each $\lambda_i$ implies that $B(v,v) = 0$ if and only if $v={\bf 0}$. 
This proves that $V_0 \cap Q = \{{\bf 0}\}$, as required. 
\end{proof}
Thus, $Q \setminus \{{\bf 0}\}$ is contained in $V \setminus W(V_0)$. 
This is also true in the case where $B$ is of positive type. 
Let $$\widehat{Q}:=\widehat{Q \setminus \{{\bf 0}\}}.$$ 
Then $\widehat{Q}$ coincides with the set $\{x \in V_1 : B(x,x)=0\}$, which has already appeared in Introduction. 
Lemma \ref{zerodake} above shows that the normalized action is also well-defined 
on $\widehat{Q} \subset V_1 \setminus W(V_0)$ everywhere. 

Moreover, we also see that $$\Delta \cap W(V_0) = \emptyset.$$ 
In fact, for any $\rho \in V^+ \cup V^-$, 
if the Euclidean inner product of $\rho$ and $o$, which coincides with $|\rho|_1$, is equal to 0, 
then $\rho$ should be {\bf 0} by Lemma \ref{heimen}. 
Since the root system $W(\Delta)$ is contained in $V^+ \cup V^-$ (see \cite[Remark 1.3]{hlr}), 
we obtain that $W(\Delta) \cap V_0=\emptyset$, i.e., $\Delta \cap W(V_0)=\emptyset$. 
This is also obvious in the case where $B$ is of positive type. 

Let $\widehat{Q_-}:=\{x \in V_1 : B(x,x) < 0\}$. 
Then, we note that the boundary of $\widehat{Q_-}$ with respect to 
the subspace topology on $V_1$ coincides with $\widehat{Q}$. 
Since $B$ is a symmetric bilinear form, we can diagonalize it 
by an orthogonal transformation $L$. Here we assume that $L o =\alpha_n$. 
Then we see that ${}^tL B L$ is equal to the diagonal matrix $(\lambda_1,\ldots,\lambda_{n-1},-\lambda_n)$, 
denoted by $A$, where $\lambda_1,\ldots,\lambda_{n-1},-\lambda_n$ are the eigenvalues of $B$ with $\lambda_i > 0$ and 
${}^t M$ means the transpose of a matrix $M$. 
Consider a basis $L^{-1} \Delta=\{\beta_1,\ldots,\beta_n\}$. 
Then $\widehat{Q_-}=\{\sum_{i=1}^nv_i\beta_i \in V: \lambda_1v_1^2+\cdots+\lambda_{n-1}v_{n-1}^2-\lambda_nv_n^2<0\} \cap V_1$. 
From the definition of $V_1$, 
we conclude that $\widehat{Q_-}$ is an ellipsoid and $\widehat{Q}$ is its boundary.

\subsection{Visibility on $\widehat{Q}$}

Next, we recall a valuable notion ``visibility'' from \cite[Section 3]{hlr} 
and discuss the visible points on $\widehat{Q}$. 

Let $L(x,y)$ (resp. $L[x,y]$) denote the {\it line} through $x$ and $y$ 
(resp. the {\it segment} joining $x$ and $y$). 
Using this, we define a valuable idea given in \cite{hlr}. 
That is, we say that $x \in \widehat{Q}$ is {\it visible} from $y \in V_1$ 
if $L[x,y] \cap \widehat{Q} = \{x\}$. 
Given $y \in V_1$, we call a curve consisting of visible points from $y$ in $\widehat{Q}$ a {\it visible curve from $y$}. 
If there is no confusion, then we simply call it a visible curve.

The set of all visible points of $\widehat{Q}$ 
from a normalized simple root $\widehat{\alpha}$ is said to be a {\em visible area} from $\widehat{\alpha}$, 
denoted by $V_\alpha$. 

We recall the following proposition concerning with the notion ``visible''. 
\begin{lem}[{\cite{hlr}}]\label{visibility}
Let $x \in \widehat{Q}$ and $\alpha \in \Delta$. 
\begin{itemize}
\item[(i)] $x \in V_\alpha$ if and only if $B(\alpha,x) \geq 0$. 
\item[(ii)] $x$ and $s_\alpha \cdot x$ lie on the same line $L(x,\alpha)$. 
\item[(iii)] $x \in \partial V_\alpha$ if and only if $B(\alpha,x)=0$. 
\end{itemize}
\end{lem}

The statements (i) and (ii) of Lemma \ref{visibility} correspond to 
\cite[Proposition 3.5 (i) and Proposition 3.7 (i)]{hlr}. 
The statement (iii) of Lemma \ref{visibility} follows from the continuity of $B$ on $\widehat{Q}$ and (i). 
Although the definition of $| \cdot |_1$ in this paper is different from that of \cite{hlr}, 
their proofs in \cite{hlr} still work since $B(x,\alpha) = |\alpha|_1B(x,\widehat{\alpha})$ 
and $|\alpha|_1>0$ hold for $\alpha \in \Delta$.

\begin{prop}\label{cover}
There is no element in $\widehat{Q} \cap \mathrm{conv}(\widehat{\Delta})$ 
which is never visible from any normalized simple root. 
In other words, $\widehat{Q} \cap \mathrm{conv}(\widehat{\Delta})$ is covered by $\{V_\alpha : \alpha \in \Delta\}$. 
\end{prop}
\begin{proof}
Let $x \in \widehat{Q} \cap \text{conv}(\widehat{\Delta})$. 
Then $x$ can be written like $x=\sum_{i=1}^n x_i \widehat{\alpha_i}$, where $x_i \geq 0$ and $\sum_{i=1}^n x_i=1$ 
and $\alpha_1,\ldots,\alpha_n$ are simple roots. Thus we have 
$$B(x,x)=B\left( x, \sum_{i=1}^n x_i \widehat{\alpha_i} \right)= \sum_{i=1}^n x_iB\left( x, \widehat{\alpha_i} \right).$$
If we suppose that $B(x, \widehat{\alpha_i})<0$ for every $i$, then $B(x,x)<0$, a contradiction. 
Thus, there is at least one $i$ such that $B(x, \alpha_i) \geq 0$. 
This implies that $x$ is visible from some normalized simple root $\widehat{\alpha_i}$ from Lemma \ref{visibility}. 
\end{proof}

In the following, we prove some lemmas for our proof of Theorem \ref{mein}. 

\begin{lem}\label{1/2}
For any $x \in \widehat{Q}$ and $\alpha \in \Delta$, one has $B(\alpha,x)<\frac{1}{2|\alpha|_1}$. 
\end{lem}
\begin{proof}
Suppose that $B(\alpha, x) \geq \frac{1}{2|\alpha|_1}$ 
for some $x \in \widehat{Q}$ and $\alpha \in \Delta$. 
Then we have $|s_\alpha(x)|_1=1 - 2B(\alpha,x)|\alpha|_1 \leq 0$. On the other hand, for $y \in \partial V_\alpha$, 
we have $B(\alpha,y)=0$ by Lemma \ref{visibility} (iii), implying that $|s_\alpha(y)|_1=1-2B(y,\alpha)|\alpha|_1=1>0$. 
Now from the continuity of a linear map $B(\alpha, *)$, there should be $z \in \widehat{Q}$ 
such that $|s_\alpha(z)|_1=0$. 
However, as mentioned in Lemma \ref{zerodake}, $W(V_0) \cap \widehat{Q} = \emptyset$, a contradiction. 
\end{proof}

\begin{prop}\label{hitoshii}
        For $x,y \in Q\setminus\{{\bf 0}\}$, 
\begin{itemize}
	\item[(a)] one has $B(x,y) = 0$ if and only if $\widehat{x} = \widehat{y}$; 
	\item[(b)] if $\widehat{x} \not= \widehat{y}$, then one has $B(\widehat{x},\widehat{y}) < 0$; 
	\item[(c)] for any $\alpha \in \Delta$ and $x,y \in V_\alpha$, we have 
        		$$|B(s_\alpha \cdot x, s_\alpha \cdot y)| \geq |B(x,y)|$$
        and the equality of this inequality holds if and only if $x,y \in \partial V_\alpha$. 
\end{itemize}
\end{prop}
\begin{proof}
	For $x,y \in Q \setminus \{{\bf 0}\}$ with $\widehat{x} \not= \widehat{y}$, 
        one has $|\widehat{x}-\widehat{y}|_1=|\widehat{x}|_1-|\widehat{y}|_1=0$. Thus, $\widehat{x}-\widehat{y} \in V_0$. 
	Note that $x \not\in V_0$ and $y \not\in V_0$ by $Q \cap V_0=\{{\bf 0}\}$ (Lemma \ref{zerodake}). 
        Moreover, it also follows that 
        $B$ is positive-definite on $V_0$. Thus, we see that 
        $B(\widehat{x}-\widehat{y},\widehat{x}-\widehat{y})=-2B(\widehat{x},\widehat{y}) \geq 0$ 
	and $B(\widehat{x}-\widehat{y},\widehat{x}-\widehat{y})=0$ if and only if $\widehat{x}-\widehat{y}={\bf 0}$. 
        Hence, we conclude 
        \begin{itemize}
        	\item[(a)] $B(x,y)=|x|_1|y|_1B(\widehat{x},\widehat{y})=0$ if and only if $\widehat{x}=\widehat{y}$; 
        	\item[(b)] $B(\widehat{x},\widehat{y}) < 0$ if $\widehat{x} \not= \widehat{y}$. 
        \end{itemize}
	
        In addition, thanks to Lemma \ref{visibility} and Lemma \ref{1/2}, 
        one has $0 \leq B(x,\alpha)|\alpha|_1, B(y, \alpha)|\alpha|_1 < \frac{1}{2}$. 
	Thus the inequality $0 < |s_\alpha(x)|_1=|x-2B(x,\alpha)\alpha|_1 \leq 1$ holds. 
	Similarly, $0 < |s_\alpha(y)|_1 \leq 1$. Hence 
	$$|B(s_\alpha \cdot x, s_\alpha \cdot y)|=\frac{|B(s_\alpha(x), s_\alpha(y))|}{|s_\alpha(x)|_1 |s_\alpha(y)|_1}
	=\frac{|B(x, y)|}{|s_\alpha(x)|_1 |s_\alpha(y)|_1} \geq |B(x,y)|.$$
        In particular, since $|s_\alpha(x)|_1=1$ if and only if we have $x \in \partial V_\alpha$, 
        $|B(x,y)| < |B(s_\alpha \cdot x, s_\alpha \cdot y)|$ holds if $x \not\in \partial V_\alpha$ or $y \not\in \partial V_\alpha$. 
\end{proof}

\subsection{A metric on $\widehat{Q}$} 

Next, we define a metric on $\widehat{Q}$ by a bilinear form $B$. 

We first remark the following. 

\begin{rem}\label{metorikku}{\em 
As in the proof of Proposition \ref{hitoshii}, 
$B$ is positive-definite on $V_0$ (not $V$) by $Q \cap V_0 = \{{\bf 0}\}$. 
Thus, $B(x-y,x-y)^{\frac{1}{2}}$ defines a metric on $V_1$. 
Moreover, since $|B(x-y,x-y)|=2|B(x,y)|$ for $x,y \in \widehat{Q}$, 
$|B(\cdot, \cdot)|^{\frac{1}{2}} : \widehat{Q} \times \widehat{Q} \to \R_{\geq 0}$ 
defines a metric on $\widehat{Q}$. We see that $\sup_{x, y \in \widehat{Q}, x \not= y}\frac{|B(x,y)|^\frac{1}{2}}{\|x-y\|}$ 
is bounded, where $\|\cdot\|$ denotes the Euclidean norm, because of the following: 
\begin{align*}
\sup_{x,y \in \widehat{Q}, x \not=y}\frac{\sqrt{2}|B(x,y)|^\frac{1}{2}}{\|x-y\|} 
&=\sup_{x,y \in \widehat{Q}, x \not=y}\frac{\sqrt{2}\left|\frac{1}{2}B(x-y,x-y)\right|^\frac{1}{2}}{\|x-y\|} \\
&\leq \sup_{v \in V_0, v \not= 0} \frac{|B(v,v)|^\frac{1}{2}}{\|v\|} = 
\sup_{v \in V_0, \|v\|=1} |B(v,v)|.
\end{align*}
Since the region $\{v \in V_0 : \|v\|=1\}$ is compact and the bilinear map $B(,)$ is continuous, 
there is $u \in V$ such that $|B(u,u)|=\sup_{v \in V_0, \|v\|=1} |B(v,v)|< \infty$. Conversely, 
we also see that $\sup_{x, y \in \widehat{Q}, x \not= y}\frac{\|x-y\|}{|B(x,y)|^\frac{1}{2}}$ is bounded 
because $B$ is positive-definite on $V_0$. These show the comparability of $|B(\cdot, \cdot)|^{\frac{1}{2}}$ and $\|\cdot\|$ on $\widehat{Q}$. 
}\end{rem}

Let $c$ be a curve in $\widehat{Q}$ connecting $x$ and $y$ for $x, y \in \widehat{Q}$. 
The {\it length} $\ell_B(c)$ of $c$ is defined by
$$
	\ell_B(c) = \sup_{C} \sum_{i=1}^n |B(x_{i-1},x_i)|^{\frac{1}{2}} ,
$$
where the supremum is taken over all chains $C = \{ x = x_0,x_1,\ldots,x_n = y \}$ on $c$ with unbounded $n$. 
Given $x, y \in \widehat{Q}$ with $x \not= y$, we define 
$$
	d_B(x,y) = \inf\{\ell_B(c):c \text{ is a curve joining }x \text{ and }y\}.
$$
Since $B$ is symmetric, the symmetry of $d_B$ is trivial. 
Moreover, the nonnegativity of $d_B$ is also trivial. 
In addition, the triangle inequality can be seen easily. 
Hence $d_B : \widehat{Q} \times \widehat{Q} \to \R_{\geq 0}$ 
is a pseudometric on $\widehat{Q}$. The following lemma guarantees that $d_B$ is a metric. 
\begin{lem}
For any $x,y \in \widehat{Q}$, if $d_B(x,y)=0$, then $x=y$.
\end{lem}
\begin{proof}
When $d_B(x,y)=0$, for an arbitrary $\epsilon > 0$, 
there exists a curve $c$ such that $\ell_B(c)<\epsilon$. 
By the definition of $\ell_B$, since $\sum_{i=1}^m|B(x_{i-1},x_i)|^{\frac{1}{2}}<\epsilon$ 
for any chain of $c$, one has $|B(x,y)|^\frac{1}{2} < \epsilon$. 
This means that $B(x,y)=0$. Therefore, $x=y$ by Proposition \ref{hitoshii}. 
\end{proof}

\begin{lem}\label{metric2}
Let $d_E$ be the length metric using $\|\cdot\|$ defined by the same way as $d_B$. 
Then the metric space $(\widehat{Q},d_B)$ is homeomorphic to $(\widehat{Q},d_E)$. 
\end{lem}
\begin{proof}
The discussions in Remark \ref{metorikku} imply the comparability of $d_E$ and $d_B$. 
In fact, the (Lipschitz) continuity of $\text{id} : (\widehat{Q},d_E) \to (\widehat{Q},d_B)$ 
can be proved as follows. For an arbitrary curve $c$ in $\widehat{Q}$ joining $x$ and $y$ 
and $\epsilon > 0$, there exists a chain $C=\{x=x_0,x_1,\ldots,x_m=y\}$ such that 
$$\ell_B(c)-\epsilon \leq \sum_{i=1}^m|B(x_{i-1},x_i)|^\frac{1}{2} \leq 
K\sum_{i=1}^m\|x_{i-1}-x_i\| \leq K \ell_E(c),$$ 
where $K=\sup_{x,y \in \widehat{Q}, x \not=y}\frac{\sqrt{2}|B(x,y)|^\frac{1}{2}}{\|x-y\|}$ 
and $\ell_E(c)$ is the length of $c$ defined by using $\| \cdot \|$. 
Thus, for any $x,y \in \widehat{Q}$ and $\epsilon > 0$, we have 
$d_B(x,y)-\epsilon \leq K d_E(x,y)$. Hence, $d_B(x,y) \leq K d_E(x,y)$. 
Similarly, we also see the (Lipschitz) continuity of $\text{id} : (\widehat{Q},d_B) \to (\widehat{Q},d_E)$. 
\end{proof}

Since $\widehat{Q}$ is the boundary of an ellipsoid in $V \cong \R^n$, 
$\widehat{Q}$ is a $C^\infty$ manifold and its topology induced from $d_E$ 
coincides with the relative topology of $V$. 
Clearly, $\widehat{Q}$ is compact on the relative topology of $V$, 
so $(\widehat{Q},d_B)$ is also compact by Remark \ref{metorikku} and Lemma \ref{metric2}. 
The compactness of $(\widehat{Q},d_B)$ also implies that $(\widehat{Q}, d_B)$ is a geodesic space 
by Hopf--Rinow Theorem (see \cite[p.9]{gromov}). 
Moreover, since each normalized simple reflection is a homeomorphism on $(\widehat{Q},\|\cdot \|)$, 
we obtain the following:

\begin{prop}
The metric space $(\widehat{Q},d_B)$ is a compact geodesic space. 
Moreover, $W$ acts on $(\widehat{Q}, d_B)$ as a homeomorphism. 
\end{prop}

Finally, we observe more precise properties of the normalized action on $(\widehat{Q},d_B)$.

\begin{prop}\label{hodai2}
	Let $\alpha \in \Delta$ and $x,y \in V_\alpha$. 
	\begin{itemize}
	\item[(i)] 		
                Each geodesic between $x$ and $y$ is contained in $V_\alpha$. 
	\item[(ii)]
		For any visible curve $c$ from $\widehat{\alpha}$, we have 
                \begin{align}\label{siki}
                	\ell_B(c) \leq \ell_B(s_\alpha \cdot c). 
                \end{align}
                Moreover, $d_B(x,y) \leq d_B(s_\alpha \cdot x, s_\alpha \cdot y)$. 
        \item[(iii)] 
         	For any non-trivial curve $c$, the equality of \eqref{siki} holds if and only if $c \subset \partial V_\alpha$. 
        \end{itemize}
\end{prop}
\begin{proof}
	Let $c'$ be a geodesic joining $x$ and $y$. Let us decompose $c'$ into 
        $$c'= \bigcup_{i \in I} c_i \cup \bigcup_{j \in J} c_j,$$ 
        where $c_i$ ($i \in I$) is a visible curve from $\alpha$ and $c_j$ ($j \in J$) the others. 
	Remark that $J$ might be empty. 
        Set $c''=\bigcup_{i \in I} c_i \cup \bigcup_{j \in J} s_\alpha \cdot c_j$. 
        Then $c''$ is a curve joining $x$ and $y$ because 
        each point of the boundary of the visible area from $\alpha$ is fixed by $s_\alpha$. 
	By Proposition \ref{hitoshii} (c), we obtain $\ell_B(c') > \ell_B(c'')$ if $J$ is not empty. 
        However, since $c'$ is a geodesic joining $x$ and $y$, $J$ should be empty. 
	This says that each geodesic between $x$ and $y$ is contained in $V_\alpha$. 
        
        Next, we prove (iii). If $c \subset \partial V_\alpha$, 
        since $B(\alpha,x)=0$ for any $x \in c$, the equality of \eqref{siki} directly follows. 
        Assume that $\ell_B(c)=\ell_B(s_\alpha \cdot c)$. 
        Then for an arbitrary curve $c' \subset c$, we also have $\ell_B(c')=\ell_B(s_\alpha \cdot c')$. 
        Decompose $c$ into $k$ curves for an arbitrary fixed $k \in \Z_{>0}$. 
        Let $c_1$ be one component of such curves. For $0 < \epsilon < 1$, by the definition of $\ell_B$, 
        there exists a chain $\{x_1,\ldots,x_m\}$, where $x_1$ and $x_m$ are the endpoints of $c_1$, 
        such that 
	$$\sum_{i=1}^m|B(x_{i-1},x_i)|^\frac{1}{2} \geq (1 - \epsilon)^\frac{1}{2}\ell_B(c_1).$$
        Since $\ell_B(c_1)=\ell_B(s_\alpha \cdot c_1)$, one has 
        $$(1 - \epsilon)^\frac{1}{2}\ell_B(c_1)=(1 - \epsilon)^\frac{1}{2}\ell_B(s_\alpha \cdot c_1)
        \geq (1 - \epsilon)^\frac{1}{2}\sum_{i=1}^m|B(s_\alpha \cdot x_{i-1}, s_\alpha \cdot x_i)|^\frac{1}{2}.$$ Hence, 
        $$\frac{\sum_{i=1}^m|B(x_{i-1},x_i)|^\frac{1}{2}}{\sum_{i=1}^m|B(s_\alpha \cdot x_{i-1},s_\alpha \cdot x_i)|^\frac{1}{2}}
        \geq (1-\epsilon)^\frac{1}{2}.$$

	Now, in general, for positive real numbers $a_1,\ldots,a_m$ and $b_1,\ldots,b_m$, 
        we see that $$\frac{\sum_{i=1}^ma_i}{\sum_{i=1}^mb_i}\leq \max_{i\in\{1,\ldots,m\}}\frac{a_i}{b_i}.$$

	Thus, there exists some $i$ such that 
        \begin{align*}
        \frac{|B(x_{i-1},x_i)|^\frac{1}{2}}{|B(s_\alpha \cdot x_{i-1},s_\alpha \cdot x_i)|^\frac{1}{2}}\geq (1-\epsilon)^\frac{1}{2}
	\Longleftrightarrow 
        \frac{|B(x_{i-1},x_i)|}{|B(s_\alpha \cdot x_{i-1},s_\alpha \cdot x_i)|}\geq 1-\epsilon \\
	\Longleftrightarrow (1-2B(x_{i-1},\alpha)|\alpha|_1)(1-2B(x_i,\alpha)|\alpha|_1) \geq 1-\epsilon. 
	\end{align*}
        On the other hand, since $x_{i-1},x_i \in V_\alpha$, one has 
        $1 - 2B(x_{i-1},\alpha)|\alpha|_1 \leq 1$ and $1 - 2B(x_i,\alpha)|\alpha|_1 \leq 1$. Hence 
        $1-\epsilon \leq 1 - 2B(x_{i-1},\alpha)|\alpha|_1 \leq 1$ and $1-\epsilon \leq 1 - 2B(x_i,\alpha)|\alpha|_1 \leq 1$. 
        Since $\epsilon$ is arbitrary, by taking $\epsilon$ as $\epsilon \to 0$, 
        we see that $x_{i-1}$ and $x_i$ belong to $\partial V_\alpha$. 
        Moreover, since $k$ is also arbitrary, by taking $k$ as $k \to \infty$, 
        we conclude that $c \subset \partial V_\alpha$, as desired. 
\end{proof}


\section{A proof of Theorem \ref{kidou}}\label{onepoint}

In this section, we prove Theorem \ref{kidou}. 
First, we note that $x \in \widehat{Q}$ is fixed by the normalized action of $s_\alpha$ 
for $\alpha \in \Delta$ if and only if $B(x,\alpha)=0$. 
\begin{rem}\label{atarimae}{\em 
In general, each point $x \in \widehat{Q}$ fixed by every normalized action of $s_\alpha$ 
corresponds to an eigenvector whose eigenvalue is 0. 
Thus, there is no such point when $B$ is definite, in particular, $B$ has the signature $(n-1,1)$. 
Hence, there is no element in $\widehat{Q}$ which is fixed by every $s_\alpha$ with $\alpha \in \Delta$. 
In other words, we have $\bigcap_{\alpha \in \Delta} \partial V_\alpha = \emptyset$. 
}\end{rem}

\begin{lem}\label{kotei}
Let $K \subset \widehat{Q}$ be a nonempty $W$-invariant subset of $\widehat{Q}$. 
Then for each $\alpha \in \Delta$, there is $x_\alpha \in K$ such that 
$x_\alpha \not= s_\alpha \cdot x_\alpha$. 
\end{lem}
\begin{proof}
By Remark \ref{atarimae}, 
$K$ contains $x$ with $x \not= s_{\alpha_0} \cdot x$ for some $\alpha_0 \in \Delta$. 

Fix $\alpha \in \Delta$ arbitrarily. Since we assume that $W$ is irreducible, 
the Coxeter graph associated with $W$ is connected (cf. \cite[Section 2.2]{Hum90}). 
Hence there is a path from $\alpha_0$ to $\alpha$ in the Coxeter graph, that is to say, 
there is a sequence of simple roots $(\alpha_0,\alpha_1,\ldots,\alpha_k)$ such that 
$\alpha_k=\alpha$ and $B(\alpha_{i-1},\alpha_i) \not= 0$ for $i=1,\ldots,k$, 
where $k$ is some positive integer.

For $i=0$, by the above discussions, 
there is a point $x_0 \in K$ such that $x_0$ is not fixed by $s_{\alpha_0}$, 
i.e., $x_0 \not= s_{\alpha_0} \cdot x_0$. For $i=1$, we have $B(\alpha_0,\alpha_1) \not= 0$. 
On the other hand, since 
\begin{align*}
B(s_{\alpha_0} \cdot x_0, \alpha_1)=\frac{B(x_0-2B(x_0, \alpha_0)\alpha_0,\alpha_1)}{|s_{\alpha_0} (x_0)|_1} 
=\frac{B(x_0,\alpha_1)-2B(x_0,\alpha_0)B(\alpha_0,\alpha_1)}{|s_{\alpha_0} ( x_0 )|_1}, 
\end{align*}
if $B(x_0,\alpha_1) = 0$, then $B(s_{\alpha_0} \cdot x_0, \alpha_1) \not=0$ 
because $B(x_0,\alpha_0) \not=0$ and $B(\alpha_0,\alpha_1)\not=0$. 
Hence either $B(x_0,\alpha_1)$ or $B(s_{\alpha_0} \cdot x_0, \alpha_1)$ is nonzero. 
This means that either $x_0$ or $s_{\alpha_0} \cdot x_0$ is not fixed by $s_{\alpha_1}$. 
Let $x_1$ be such point. Here we remark that since $x_0 \in K$ and $K$ is $W$-invariant, 
we know that $s_{\alpha_0} \cdot x_0 \in K$, so $x_1$ belongs to $K$. 
Similarly, we obtain that either $x_1$ or $s_{\alpha_1} \cdot x_1$ is not fixed by 
$s_{\alpha_2}$. Let $x_2$ be such point. By repeating this consideration, 
we eventually obtain $x_k \in K$ such that $x_k$ is not fixed by $s_\alpha$, as required. 
\end{proof}

The following proposition plays a crucial role in the proof of Theorem \ref{kidou}. 

\begin{prop}\label{keyprop}
The set $E$ of accumulation points of normalized roots is a minimal $W$-invariant closed set. 
\end{prop}
\begin{proof}
Let $K \subset \widehat{Q}$ be a $W$-invariant closed subset. We may show that $E \subset K$. 

Let $$K'=\left(\bigcup_{x,y \in K, x \not= y} L(x,y) \setminus W(V_0) \right) \setminus \widehat{Q_-}.$$ Thus $K'$ is also a $W$-invariant set. 
In what follows, we claim the inclusion $$\overline{K'} \cap \widehat{Q} \subset K,$$ 
where $\overline{K'}$ is the closure of $K'$ with respect to the Euclidean topology on $V$. 
Suppose, on the contrary, that $\overline{K'} \cap \widehat{Q} \not\subset K$, i.e., 
there is $p \in \overline{K'} \cap \widehat{Q}$ such that $p \in \widehat{Q} \setminus K$.
Then there is an open neighborhood of $p$ in $\widehat{Q}$ such that $U \subset \widehat{Q} \setminus K$ 
because $K$ is closed. Take a sequence $\{p_i\}$ in $K'$ converging to $p$. 
By the definition of $K'$, there are $x_i$ and $y_i$ in $K$ with $x_i \not= y_i$ such that $p_i \in L(x_i,y_i)$ for each $i$. 
We fix such $x_i$ and $y_i$ and assume that $x_i$ is visible from $p_i$ and $y_i$ is not. 
Let $v_i=\frac{1}{2}p_i+\frac{1}{2}x_i$. 
Then we have $B(v_i,v_i)=\frac{1}{4}(B(p_i,p_i)+2B(p_i,x_i))$. Since $p_i$ converges to $p \in \widehat{Q}$, $B(p_i,p_i)$ goes to 0. 
However, since $x_i$ never goes to $p$, we have $B(p_i,p_i) \leq |B(p_i,x_i)|$ for sufficiently large $i$. 
(Otherwise, by taking a subsequence $\{p_{j_k}\}$ of $\{p_i\}$ with $B(p_{j_k},p_{j_k}) > |B(p_{j_k},x_{j_k})|$ for each $k$, 
one has $0 \leq |B(p_{j_k},x_{j_k})| < B(p_{j_k},p_{j_k}) \rightarrow 0$, which means that $x_{j_k}$ goes to $p$, a contradiction.) 
Moreover, since $x_i$ never goes to $p$ again, one has $B(p_i,x_i)<0$ for sufficiently large $i$ 
by Proposition \ref{hitoshii}. Hence, for sufficiently large $i$, we have 
$$B(v_i,v_i)=\frac{1}{4}(B(p_i,p_i)+2B(p_i,x_i)) \leq \frac{1}{4}(|B(p_i,x_i)|+2B(p_i,x_i)) \leq \frac{B(p_i,x_i)}{4} \leq 0.$$
On the one hand, since $x_i$ is visible from $p_i$, each point in $L[x_i,p_i] \setminus \{x_i,p_i\}$ 
does not belong to $\widehat{Q} \cup \widehat{Q_-}$. In particular, $v_i \not\in \widehat{Q} \cup \widehat{Q_-}$. 
Thus, we have $B(v_i,v_i)>0$ for every $i$, a contradiction. Therefore, 
we conclude the desired inclusion $\overline{K'} \cap \widehat{Q} \subset K$.

For each $\alpha \in \Delta$, when we take $x \in K$ with $x \not= s_\alpha \cdot x$, 
$L(x,s_\alpha \cdot x)$ intersects with $\alpha$. Since $x$ and $s_\alpha \cdot x$ belong to $K$, 
$\alpha$ belongs to $K'$. 
By Lemma \ref{kotei}, we can take such an element of $K$ for every $\alpha \in \Delta$. 
Hence $\Delta \subset K'$. Since $K'$ is $W$-invariant, we also have $W \cdot \Delta = \widehat{\Phi} \subset K'$. 
Thus, the accumulation points of $W \cdot \Delta$, which is nothing but $E$, should be contained in 
$\overline{K'} \cap \widehat{Q} \subset K$. 
Therefore, we have $E \subset K$. 
\end{proof}

\begin{proof}[Proof of Theorem \ref{kidou}]
For any $x \in E$, it is obvious that $\overline{W \cdot x} \subset E$. 
Moreover, since $\overline{W \cdot x}$ is $W$-invariant closed set, from Proposition \ref{keyprop}, 
we also have $E \subset \overline{W \cdot x}$, as desired. 
\end{proof}


\section{Fixed points of the normalized action}\label{koteiten}

For $w \in W$, there exists a constant $C \geq 1$ such that for any $x \in Q$, we have 
\begin{align}\label{inequality}
C^{-1} ||x|_1| \le ||w(x)|_1| \le C ||x|_1|.
\end{align}
We may choose a constant $C$ such that $C$ is independent of the choice of $x$. 
In fact, from $Q \setminus \{{\bf 0}\} \subset V \setminus W(V_0)$, 
we can compute 
$$\sup_{x \in Q \setminus \{{\bf 0}\}}\left|\frac{|w(x)|_1}{|x|_1}\right|=
\sup_{x \in Q \setminus \{{\bf 0}\}}\left|\left|w\left(\frac{x}{|x|_1}\right)\right|_1\right|=
\sup_{y \in \widehat{Q}}||w(y)|_1|.$$ 
Since $\widehat{Q}$ is compact, there exists $y' \in \widehat{Q}$ such that 
$||w(y')|_1|=\sup_{y \in \widehat{Q}}||w(y)|_1|$. 
Similarly, there also exists $y'' \in \widehat{Q}$ such that $||w(y'')|_1|=\inf_{y \in \widehat{Q}}||w(y)|_1|$. 
Let $C=\max\left\{||w(y')|_1|, \frac{1}{||w(y'')|_1|}\right\}$. Thus $C \geq 1$ and 
the inequality \eqref{inequality} is satisfied. 

\begin{lem}\label{zureta}
	For $w \in W$ with infinite order and $x \in \widehat{Q}$,
	let $(w^{n_i}\cdot x)_{n_i}$ be a converging subsequence of $(w^n\cdot x)_n$ 
	to $y \in \widehat{Q}$.
	If $||w^{n_i}(x)|_1| \rightarrow \infty$, then for any $k \in \Z$
	the sequence $(w^{n_i+k} \cdot x)_{n_i}$ also converges to $y$.
\end{lem}

\begin{proof}
	Fix $k \in \Z$ arbitrarily. 
	By the remark above, we have a constant $C_k \ge 1$, which depends only on $k$, 
        so that for each $n \in \N$,
	$$
		{C_k}^{-1} ||w^n(x)|_1| \le ||w^{n+k}(x)|_1| \le C_k ||w^n(x)|_1|. 
	$$
	Then we see that
	\begin{align*}
		|B(w^{n_i}\cdot x,w^{n_i+k}\cdot x)| 
			&= \frac{|B(w^{n_i}(x),w^{n_i+k}(x))|}{||w^{n_i}(x)|_1|\cdot||w^{n_i+k}(x)|_1|}
			\ \\
			&= \frac{|B(x,w^{k}(x))|}{||w^{n_i}(x)|_1|\cdot||w^{n_i+k}(x)|_1|}
			\le C_k\frac{|B(x,w^{k}(x))|}{||w^{n_i}(x)|_1|^2}
			\rightarrow 0,
	\end{align*}
	as $n_i \rightarrow \infty$.
        By Proposition \ref{hitoshii}, we have the conclusion.
\end{proof}

For any $w \in W$, 
if there is a fixed point on $\widehat{Q}$ of the normalized action of $w$, 
then such point is an eigenvector of $w$ corresponding to a real eigenvalue. 

\begin{lem}\label{hutatsu}
	Let $w \in W$. 
	Suppose that $w$ has distinct eigenvectors $p, p'$ lying on $\widehat{Q}$,
	and let $\lambda,\lambda' \in \R$ be corresponding eigenvalues respectively. 
	Then $\lambda\lambda' = 1$.
\end{lem}

\begin{proof}
	We see this by calculating directly;
	\begin{align*}
		B(p,p') = B(w(p),w(p')) = \lambda\lambda'B(p,p').
	\end{align*} 
	Since $p$ and $p'$ are distinct and sitting on $\widehat{Q}$, 
	we have $B(p,p') \not= 0$ by Proposition \ref{hitoshii}. Hence $\lambda\lambda' = 1$, as required. 
        \end{proof}

This lemma gives us the following observations about eigenvalues satisfying the condition 
``different from $\pm1$ and corresponding eigenvectors are contained in $\widehat{Q}$''.
\begin{itemize}
	\item
	There are at most two such eigenvalues.
	\item
	The intersection of $Q$ (not $\widehat{Q}$) and each eigenspace of such eigenvalue 
	is one dimensional.
	\item
	If such eigenvalue exists, 
	there are no eigenvectors in $\widehat{Q}$ corresponding to eigenvalues of $\pm 1$.
\end{itemize}

\begin{prop}\label{meidai}
	Let $w \in W$ with infinite order and take $x \in \widehat{Q}$ arbitrarily. 
	If $w$ has an eigenvector $p$ in $\widehat{Q}$ corresponding to the eigenvalue $|\lambda| > 1$,
	then $(w^n \cdot x)_n$ converges to $p$. In particular, $p$ lies in $E$.
\end{prop}

\begin{proof}
	Since $\widehat{Q}$ is compact, there exists a converging subsequence 
	$(w^{n_i}\cdot x)_{n_i}$ of $(w^n\cdot x)_n$.
	Let $y$ be the convergent point of the sequence above.
	
	Notice that $w$ has two eigenvectors in $\widehat{Q}$ when $\lambda \not= \pm 1$.
	In such case, we denote the other eigenvector by $p'$ 
	and the corresponding eigenvalue by $\lambda'$.
	By Lemma \ref{hutatsu}, $\lambda\lambda' = 1$ must be satisfied.

	We consider an eigenvalue $\lambda'$ with $|\lambda'| < 1$ and the corresponding eigenvector $p'$. 
	Then we have 
	\begin{align*}
		|B(w^{n_i}\cdot x,p')| 
			&= \frac{|B(w^{n_i}(x),p')|}{||w^{n_i}(x)|_1|} 
				= \frac{|B(x,w^{-n_i}(p'))|}{||w^{n_i}(x)|_1|}
			= \frac{|\lambda^{n_i}|\cdot|B(x,p')|}{||w^{n_i}(x)|_1|} 
	\end{align*}
	and $|B(w^{n_i}\cdot x, p')| \rightarrow |B(y,p')|$ as $n_i \rightarrow \infty$.
	Since $|\lambda^{n_i}| \rightarrow \infty$, we have $||w^{n_i}(x)|_1| \rightarrow \infty$.
	Now there exists a constant $C\ge 1$ so that
	$$
		C^{-1} ||x|_1| \le ||w(x)|_1| \le C ||x|_1|, 
	$$
	which is independent of $x$. Therefore, 
	\begin{align*}
		|B(w^{n_i}\cdot x,w\cdot y)| 
			&= \frac{|B(w^{n_i}(x),w(y))|}{||w^{n_i}(x)|_1|\cdot ||w(y)|_1|} 
				= \frac{|B(w^{n_i-1}(x),y)|}{||w^{n_i}(x)|_1|\cdot ||w(y)|_1|}\\
			&\leq \frac{C|B(w^{n_i-1}(x),y)|}{||w^{n_i-1}(x)|_1|\cdot ||w(y)|_1|}
                        =\frac{C}{||w(y)|_1|}|B(w^{n_i-1} \cdot x,y)|\rightarrow 0,	
	\end{align*}
	as $n_i \rightarrow \infty$ by Lemma \ref{zureta}.
	This implies that $y = w \cdot y$.
	Since we can apply this argument for each converging subsequence of $(w^n\cdot x)_n$ 
        and its convergent point, we can deduce that the convergent point, say, $y$, is fixed by $w$. 
	By Lemma \ref{hutatsu}, we have the following two possibilities:
	\begin{itemize}
		\item[1)]
			$w$ has only one fixed point in $\widehat{Q}$; 
		\item[2)]
			$w$ has two fixed points $p,p'$ in $\widehat{Q}$. 
	\end{itemize}
	If 1), then it is obvious that $y=p$. If 2), 
        since $|\lambda^{-1}| < 1$ and $||w^{n_i}(x)|_1| \rightarrow \infty$ as $n_i \rightarrow \infty$, one has 
       	\begin{align*}
		|B(w^{n_i}\cdot x,p)| 
			&= \frac{|B(w^{n_i}(x),p)|}{||w^{n_i}(x)|_1|} 
				= \frac{|B(x,w^{-n_i}(p))|}{||w^{n_i}(x)|_1|}\\
			&= \frac{|{\lambda}^{-n_i}|\cdot|B(x,p)|}{||w^{n_i}(x)|_1|} \rightarrow 0.
	\end{align*}
	In both cases 1) and 2), all converging subsequences of $(w^n \cdot x)_n$ converge to the same point. 
        Thus we have the conclusion.
\end{proof}

Let $w \in W$ and $x \in \widehat{Q}$ be elements satisfying the condition 
in the claim of Proposition \ref{meidai}. 
Then we have a converging sequence $(w^{n_i}\cdot x)_{n_i}$ to $y \in \widehat{Q}$ 
so that $|w^{n_i}(x)|_1 \rightarrow \infty$. 
Even in the case of $\lambda=\pm 1$, 
one has $$\left|B(w^{n_i} \cdot x, p)\right| = \left|\frac{B(w^{n_i}(x), p)}{|w^{n_i}(x)|_1}\right| 
= \left|\frac{B(x, p)}{|w^{n_i}(x)|_1}\right|.$$ 
This shows that if $||w^{n_i}(x)|_1| \rightarrow \infty$, 
then $|B(w^{n_i} \cdot x, p)|$ converges to 0. 
This means that $(w^{n_i} \cdot x)_{n_i}$ converges to $p$ by Proposition \ref{hitoshii} (a). 

\begin{rem}\label{chuui}{\em 
In the case where $W$ is rank 3, for $\delta_1,\delta_2 \in \Delta$, 
\begin{itemize}
\item if $B(\delta_1,\delta_2)<-1$, then there exist two real eigenvalues of 
$s_{\delta_1}s_{\delta_2}$ which are distinct from $\pm 1$; 
\item if $B(\delta_1,\delta_2)=-1$, then an easy calculation shows that 
$||(s_{\delta_1}s_{\delta_2})^n(v)|_1| \rightarrow \infty$ for any 
$v \in \widehat{Q} \setminus \{\frac{1}{2}\delta_1+\frac{1}{2}\delta_2\}$. 
\end{itemize}

}\end{rem}


\section{A proof of Theorem \ref{mein} : the case of rank 3}\label{syoumei}

We first prove Theorem \ref{mein} for the case of rank 3. 
Since we can handle this case more combinatorially than the case of higher ranks, 
we concentrate on this case in this section. 
Before proving, we remark that the case of rank 2 is self-evident. 
\begin{rem}\label{rima-ku}{\em 
We consider the case of rank 2. Let $\alpha, \beta$ be two simple roots. 
\begin{itemize}
\item If $B(\alpha,\beta)>-1$, then $\widehat{Q}$ is empty and so is $E$. 
\item If $B(\alpha,\beta)=-1$, then $E=L(\widehat{\alpha},\widehat{\beta}) \cap \widehat{Q}$ consists of one point. 
\item If $B(\alpha,\beta)<-1$, then $E=L(\widehat{\alpha},\widehat{\beta}) \cap \widehat{Q}$ consists of two points. 
\end{itemize}
For more details, see \cite[Example 2.1]{hlr}. 
}\end{rem}

Let $(W,S)$ be a Coxeter system of rank 3 with $S=\{s_\alpha, s_\beta, s_\gamma\}$ 
and $\Delta=\{\alpha, \beta, \gamma\}$ its simple system.

We fix one arbitrary accumulation point $\rho$ of orbits of normalized roots.

\begin{lem}\label{mouyamete1}
Let $\delta, \delta' \in \{\alpha,\beta,\gamma\}$ with $\delta \not= \delta'$ 
and assume that $B(\delta,\delta') \geq -1$. 
Then there exists a singleton $A_{\delta''} \subset E$, 
where $\delta'' \in \{\alpha,\beta,\gamma\} \setminus \{\delta, \delta'\}$, 
such that the point in $A_{\delta''}$ is visible from $\widehat{\delta}$ and $\widehat{\delta'}$ 
but not visible from $\widehat{\delta''}$. 
\end{lem}
\begin{proof}
When $B(\delta,\delta')=-1$, let $\eta=\frac{1}{2}\delta+\frac{1}{2}\delta'$. 
Then it is easy to see that $\widehat{\eta}=\lim_{n \rightarrow \infty} (s_\delta s_{\delta'})^n \cdot \rho \in E$. 
Moreover, we have $B(\widehat{\eta},\widehat{\delta}) = B(\widehat{\eta}, \widehat{\delta'})=0$. 
Thus $\widehat{\eta}$ is visible from $\widehat{\delta}$ and $\widehat{\delta'}$ by Lemma \ref{visibility} (i). 
In addition, by Remark \ref{atarimae}, $\widehat{\eta}$ is not visible from $\widehat{\delta''}$. 
Let $A_{\delta''}=\{\widehat{\eta}\}$. Then $A_{\delta''}$ satisfies the desired properties. 

Assume $B(\delta,\delta') > -1$. Then the order of $s_{\delta}s_{\delta'}$ is finite. 
Let $m$ be a positive integer such that $(s_\alpha s_\beta)^m=1$. 
Thus the order of the parabolic subgroup $W'$ generated by $s_\alpha$ and $s_\beta$ is $2m$. 
Let $T=W' \cdot \rho$ and let $T_\alpha \subset T$ (resp. $T_\beta \subset T$) 
be the set of the points in $T$ which are visible from $\alpha$ (resp. $\beta$). 
We prove that there is $\widehat{\eta} \in \widehat{Q}$ such that $\{\widehat{\eta}\}$ satisfies the required property. 

	Suppose that there does not exist $\widehat{\eta}$, i.e., $T_\alpha \cap T_\beta = \emptyset$. 
	Then $T_\alpha$ and $T_\beta$ have the same cardinality. 
	In fact, $s_\beta : T_\alpha \to T_\beta$ is well-defined by $T_\alpha \cap T_\beta = \emptyset$ and Lemma \ref{visibility} (ii). 
        Moreover, this is injective. Thus, one has $|T_\alpha| \leq |T_\beta|$. Similarly, $|T_\beta| \leq |T_\alpha|$. 
        Hence $|T_\alpha|=|T_\beta|$, denoted by $m'$. 
        Moreover, since $s_\alpha : T \setminus T_\alpha \to T_\alpha$ is injective, 
        one has $|T_\alpha| \geq |T \setminus T_\alpha|=|T|-|T_\alpha|$. Thus, $|T| \leq 2|T_\alpha|=|T_\alpha| + |T_\beta| \leq |T|$. 
	Hence $T \setminus (T_\alpha \cup T_\beta) =\emptyset$. 
	We write $T_\alpha = \{\rho_1,\ldots,\rho_{m'}\}$ and $T_\beta=\{\rho'_1,\ldots,\rho'_{m'}\}$. 
	Observe that $s_{\alpha}$ and $s_{\beta}$ act on $\{1,\ldots,m'\}$ as permutations 
	$\sigma_{\alpha},\sigma_{\beta} : \{1,\ldots,m'\} \rightarrow \{1,\ldots,m'\}$ 
	so that $s_\beta \cdot T_\alpha = \{\rho'_{\sigma_{\alpha}(1)},\ldots,\rho'_{\sigma_{\alpha}(m')}\}$ and 
	$s_\alpha \cdot T_\beta = \{\rho_{\sigma_{\beta}(1)},\ldots,\rho_{\sigma_{\beta}(m')}\}$.
	In particular, we recognize that each image of the points in $T_\beta$ by $s_{\alpha}$ 
	must be visible from $\alpha$ and vice versa.
	Moreover, the permutation $\sigma_{\alpha\beta} := \sigma_\alpha\sigma_\beta$ 
	has order $m'$. By Proposition \ref{hodai2} (i), each $B$-reflection extends the length of visible curves. 
        Thus we see that
	\begin{align*}
		d_B(\rho_{i},\rho_{j}) 
			&\ge d_B(\rho_{\sigma_{\alpha\beta}(i)},\rho_{\sigma_{\alpha\beta}(j)})\\
			&\ge d_B(\rho_{\sigma_{\alpha\beta}^2(i)},\rho_{\sigma_{\alpha\beta}^2(j)})\\
			&\ge \cdots \\
			&\ge d_B(\rho_{\sigma_{\alpha\beta}^{m'}(i)},\rho_{\sigma_{\alpha\beta}^{m'}(j)})
			 = d_B(\rho_{i},\rho_{j}),
	\end{align*}
	for any $i,j \in \{1,2,\ldots,m'\}$. Thus all the equalities of these inequalities must be satisfied. 
	Since $\rho_{\sigma_{\alpha\beta}(i)}=s_\alpha s_\beta \cdot \rho_i$, $d_B(\rho_i,\rho_j)=d_B(s_\beta \cdot \rho_i, s_\beta \cdot \rho_j)$ 
        also holds. From Proposition \ref{hodai2} (iii), $d_B(\rho_i,\rho_j)=d_B(s_\beta \cdot \rho_i, s_\beta \cdot \rho_j)$ 
        implies that $\rho_i$ and $\rho_j$ should belong to $\partial V_\beta \cap T \subset T_\beta$ if $\rho_i \not= \rho_j$, 
        while $\rho_i$ and $\rho_j$ belong to $T_\alpha$. 
        However, since $T_\alpha \cap T_\beta = \emptyset$, this cannot happen. Hence, $\rho_i=\rho_j$. 
        In particular, each of $T_\alpha$ and $T_\beta$ consists of one element. 
        Let $T_\alpha=\{\rho_1\}$ and $T_\beta=\{\rho_1'\}$. Then four points 
        $\alpha, \rho_1, \rho_1'$ and $\beta$ should lie in the same line $L(\alpha, \beta)$. 
        On the other hand, since $W'$ is finite, $L(\alpha,\beta)$ does not intersect with $\widehat{Q}$ 
        (see Remark \ref{rima-ku}), a contradiction. 
        Hence $T_\alpha \cap T_\beta$ is not empty. This means that $\widehat{\eta}$ exists in $W' \cdot \rho$. 
	Thus $\widehat{\eta} \in E$. 

Let $\widehat{\eta}=x_\alpha \widehat{\alpha} + x_\beta \widehat{\beta} + x_\gamma \widehat{\gamma}$. 
Since $\widehat{\eta} \in E \subset \mathrm{conv}(\widehat{\Delta})$, 
we have $0 \leq x_\alpha,x_\beta, x_\gamma \leq 1$ and $x_\alpha+x_\beta+x_\gamma=1$. Since 
$$B(\widehat{\eta},\widehat{\eta})=
x_\alpha B(\widehat{\eta},\widehat{\alpha})+x_\beta B(\widehat{\eta},\widehat{\beta})
+x_\gamma B(\widehat{\eta},\widehat{\gamma})=0, \; B(\widehat{\eta},\widehat{\alpha}) \geq 0 \text{ and }B(\widehat{\eta},\widehat{\beta}) \geq 0,$$ 
we have $B(\widehat{\eta},\widehat{\gamma}) < 0$ or 
$x_\alpha B(\widehat{\eta},\widehat{\alpha})=x_\beta B(\widehat{\eta},\widehat{\beta})=x_\gamma B(\widehat{\eta},\widehat{\gamma})=0$. 
Suppose that the latter case happens. Since $\widehat{\eta} \not= \widehat{\gamma}$, 
(a) both of $x_\alpha$ and $x_\beta$ are positive, 
or (b) either $x_\alpha$ or $x_\beta$ is 0 and the other is positive. 
\begin{itemize}
\item[(a)] Suppose that both of $x_\alpha$ and $x_\beta$ are positive. 
If $x_\gamma >0$, then we have $B(\widehat{\eta},\alpha)=B(\widehat{\eta},\beta)=B(\widehat{\eta},\gamma)=0$, a contradiction to Remark \ref{atarimae}. 
If $x_\gamma = 0$, then $\widehat{\eta} \in E \cap \mathrm{conv}(\{\widehat{\alpha},\widehat{\beta}\})$, 
a contradiction to the finiteness of $W'$. 
\item[(b)] Suppose that either $x_\alpha$ or $x_\beta$ is 0 and the other is positive, 
say, $x_\alpha>0$ and $x_\beta=0$. Then $x_\gamma > 0$ because of $\widehat{\eta} \not= \widehat{\alpha}$. 
Thus $\widehat{\eta} \in \mathrm{conv}(\{\widehat{\alpha},\widehat{\gamma}\})$. 
Since $$B(\widehat{\eta}, \beta)= x_\alpha B( \widehat{\alpha},\beta) + x_\gamma B(\widehat{\gamma},\beta) \geq 0,$$ 
we have $B(\alpha,\beta)=B(\gamma,\beta)=0$. This is a contradiction to our assumption ``irreducible''. 
\end{itemize}
\end{proof}

If $B(\delta,\delta') < -1$, then $\mathrm{conv}(\{\widehat{\delta},\widehat{\delta'}\})$ and $\widehat{Q}$ intersect 
at two points (Remark \ref{rima-ku}) and $\mathrm{conv}(\{\widehat{\delta},\widehat{\delta'}\})$ separates $\widehat{Q}$ 
into two components. Let $D$ be one of such open components with $D \cap \mathrm{conv}(\widehat{\Delta}) =\emptyset$.

\begin{lem}\label{mouyamete2}
Let $\delta, \delta' \in \{\alpha,\beta,\gamma\}$ with $\delta \not= \delta'$ and assume that $B(\delta,\delta')<-1$. 
Let $A_{\delta''} \subset \widehat{Q}$ be the closure of $D$, 
where $\delta'' \in \{\alpha,\beta,\gamma\} \setminus \{\delta, \delta'\}$ and $D$ is the set defined above. 
Then $A_{\delta''}$ satisfies the following: 
\begin{itemize}
\item the end points are contained in $E$; 
\item one end point of $A_{\delta''}$ is visible from $\widehat{\delta}$, the other is visible from $\widehat{\delta'}$ 
and $A_{\delta''}$ is not visible from $\widehat{\delta''}$. 
\end{itemize}
\end{lem}
\begin{proof}
The line joining a point in $A_{\delta''}$ and $\widehat{\delta''}$ always crosses $\widehat{Q} \setminus A_{\delta''}$. 
This means that $\text{int}(A_{\delta''})$ is not visible from $\widehat{\delta''}$, 
where $\text{int}(\cdot)$ denotes the interior relative to $\widehat{Q}$. 
By Remark \ref{rima-ku}, it follows that both two end points of $A_{\delta''}$ are contained in $E$ such that 
one end point of $A_{\delta''}$ is visible from $\widehat{\delta}$ and the other is visible from $\widehat{\delta'}$. 
Moreover, each of both two end points $\widehat{\eta}$ is contained in $\mathrm{conv}(\{\widehat{\delta},\widehat{\delta'}\})$, 
so it is written like $\widehat{\eta}=r\widehat{\delta}+(1-r)\widehat{\delta'}$, where $0 < r < 1$. Then it follows that 
$B(\widehat{\eta}, \widehat{\delta''})=rB(\widehat{\delta},\widehat{\delta''})+(1-r)B(\widehat{\delta'},\widehat{\delta''}) \leq 0$. 
Our assumption ``irreducible'' says that either $B(\widehat{\delta},\widehat{\delta''})$ 
or $B(\widehat{\delta'},\widehat{\delta''})$ is negative. 
Thus we obtain $B(\widehat{\eta}, \widehat{\delta''}) < 0$, i.e., 
$\widehat{\eta}$ is not visible from $\widehat{\delta''}$ by Lemma \ref{visibility} (i). 
Hence $A_{\delta''}$ satisfies the required property. 
\end{proof}

Now we prove the desired assertion. In the proof, we always use the relative topology of $\widehat{Q}$. 

\begin{proof}[Proof of Theorem \ref{mein} in the case of rank 3]
Fix three sets $A_\alpha, A_\beta$ and $A_\gamma$ in the statements of 
Lemma \ref{mouyamete1} and Lemma \ref{mouyamete2}. For our proof, we introduce the following notation. 
\begin{itemize}
\item Let $C_\gamma \subset \widehat{Q}$ 
be the connected closed set joining end points of $A_\alpha$ and $A_\beta$ satisfying $C_\gamma \subset V_\gamma$. 
Such a set $C_\gamma$ is uniquely determined. Similarly, we also define $C_\alpha$ and $C_\beta$. 
Note that the end points of each $C_\delta$ ($\delta \in \Delta$) are some points in $E$. 
\end{itemize}
Note that we have the equality 
\begin{align}\label{decomp}
\widehat{Q} = A_\alpha \sqcup A_\beta \sqcup A_\gamma \sqcup \mathrm{int}(C_\alpha) \sqcup \mathrm{int}(C_\beta) \sqcup \mathrm{int}(C_\gamma).
\end{align}

Suppose, on the contrary, that $E \subsetneq \widehat{Q} \setminus (\bigcup_{i=1}^m W \cdot D_i)$. 
Since $E$ is closed, there exists a connected open set $U_1 \subset (\widehat{Q} \setminus (\bigcup_{i=1}^m W \cdot D_i)) \setminus E$. 
Note that the boundary of $U_1$ consists of two points. 
By taking $U_1$ as a maximal one, we may assume that $\overline{U_1} \cap E$ is non-empty and consists of two points $a_1,b_1$ 
which are the end points of $\overline{U_1}$. 
Since the end points of $C_\alpha,C_\beta$ and $C_\gamma$ are contained in $E$, 
$\overline{U_1} \subset C_\alpha$ or $\overline{U_1} \subset C_\beta$ or $\overline{U_1} \subset C_\gamma$ occurs. 
Moreover, by $C_\delta \subset V_\delta$ for $\delta \in \{\alpha,\beta,\gamma\}$ and Proposition \ref{hodai2} (i), 
$\overline{U_1}$ is a geodesic. 

Let, say, $\overline{U_1} \subset C_\gamma$. Let $U_2 = s_\gamma \cdot U_1$ and let $a_2$ and $b_2$ be the end points of $U_2$. 
Then $U_2 \cap D_i=\emptyset$ for any $i$ by definition of $U_2=s_\gamma \cdot U_1$. Namely, $U_2 \cap A_\delta=\emptyset$ for any $\delta \in \{\alpha,\beta,\gamma\}$. 
By \eqref{decomp} together with $U_2=s_\gamma \cdot U_1 \not\subset C_\gamma \subset V_\gamma$, 
$U_2 \subset \mathrm{int}(C_\alpha)$ or $U_2 \subset \mathrm{int}(C_\beta)$ occurs. 
Moreover, $\overline{U_2} \cap E=\{a_2,b_2\}$. 
From Proposition \ref{hodai2} (i) and (iii), we notice that $\ell_B(U_1) < \ell_B(U_2)$.

Similarly, for each $n \geq 1$, if $U_n \subset \mathrm{int}(C_\delta)$ for some $\delta \in \{\alpha,\beta,\gamma\}$, 
then let $U_{n+1}=s_{\delta} \cdot U_n$ and let $a_{n+1}$ and $b_{n+1}$ be the end points of $\overline{U_{n+1}}$. 
Moreover, we also have $\ell_B(U_n) < \ell_B(U_{n+1})$. In particular, $U_i \not= U_j$ for any $i$ and $j$ with $i \not= j$. 
In addition, $\overline{U_n} \cap E = \{a_n,b_n\}$, where $a_n$ and $b_n$ are the end points of $\overline{U_n}$. 
If $U_i \cap U_j \not=\emptyset$ for some $i$ and $j$ with $i \not= j$, 
then an end point of $U_i$ belongs to $U_j$ or an end point of $U_j$ belongs to $U_i$. Since each end point of $U_i$ and $U_j$ 
is an element of $E$, we obtain $U_i \cap E \not=\emptyset$ or $U_j \cap E \not=\emptyset$, a contradiction. 
Hence $U_i \cap U_j = \emptyset$ for all $i$ and $j$ with $i \not= j$. 

Now it follows that one of $C_\alpha$, $C_\beta$ and $C_\gamma$, say, $C_\gamma$, 
contains infinitely many open sets $U_n$. Let $U_{i_k} \subset \text{int}(C_\gamma)$ 
for $k \geq 1$, where $i_1 < i_2 < i_3 < \cdots $. 
Since $U_{i_k} \cap U_{i_{k'}}=\emptyset$ for any $k \not= k'$, the disjoint union 
$\bigsqcup_{k=1}^\infty U_{i_k}$ is contained in $\text{int}(C_\gamma)$. 
On the one hand, we have $\ell_B(C_\gamma)<\infty$. 
On the other hand, since we have $0<\ell_B(U_{i_1}) < \ell_B(U_{i_2}) < \cdots $, 
we obtain that $\sum_{k=1}^\infty \ell_B(U_{i_k}) = \infty$, a contradiction.

Therefore, we conclude that $E=\widehat{Q} \setminus \bigcup_{i=1}^m W \cdot D_i$, as required. 
\end{proof}

\section{A proof of Theorem \ref{mein} : the case of an arbitrary rank}\label{higher}

Finally, in this section, 
we prove Theorem \ref{mein} for the case of an arbitrary rank. 

We divide the proof of Theorem \ref{mein} into the following two cases: 
\begin{itemize}
\item[{\bf (a)}] $\widehat{Q} \subset \text{int}(\text{conv}(\widehat{\Delta}))$; 
\item[{\bf (b)}] $\widehat{Q} \not\subset \text{int}(\text{conv}(\widehat{\Delta}))$. 
\end{itemize}
Here $\text{int}(\cdot)$ denotes the relative interior.

\subsection{The case (a)}

Since $|B(x-y,x-y)|^{\frac{1}{2}}$ is a metric on $\widehat{Q}$ (Remark \ref{metorikku}), 
for the proof of Theorem \ref{mein} in the case {\bf (a)}, we estimate $|B(x,y)|$ for $x,y \in \widehat{Q}$. 

\begin{lem}\label{hani}
There exists a constant $C'>0$ such that for any $x \in \widehat{Q}$, 
one has $B(x,\alpha) \geq C'$ for some $\alpha \in \Delta$. 
\end{lem}
\begin{proof}
By Proposition \ref{cover}, 
$\widehat{Q}$ is covered by $\{V_\alpha : \alpha \in \Delta\}$. 
Note that $V_\alpha$ is a closed set. 
Since $\widehat{Q} \subset \text{int}(\text{conv}(\widehat{\Delta}))$, 
for any $y = \sum_{i=1}^n y_i \alpha_i \in \widehat{Q}$, one has $y_i >0$. 

Suppose that there is $x \in \widehat{Q}$ such that $x \not\in \bigcup_{\alpha \in \Delta} \text{int}(V_\alpha)$. 
This means from Proposition \ref{cover} that $x$ should belong to $\bigcap_{i=1}^k \partial V_{\alpha_{q_i}}$ 
for some $\alpha_{q_1},\ldots,\alpha_{q_k} \in \Delta$, where $k<n$ by Lemma \ref{kotei}. 
Hence $B(x, \alpha_{q_i})=0$ for $i=1,\ldots,k$ by Lemma \ref{visibility} (iii). 
Moreover, $B(x,\alpha')<0$ for all $\alpha' \in \Delta \setminus \{\alpha_{q_1},\ldots,\alpha_{q_k}\}$. 
On the other hand, when $x$ can be written like $x=\sum_{i=1}^n x_i \alpha_i$, 
one has $B(x,x)=0$ from $x \in Q$, while by $x \not\in \bigcup_{\alpha \in \Delta} \text{int}(V_\alpha)$, 
one has $B(\alpha,x)=\sum_{i=1}^n x_iB(\alpha,\alpha_i)<0$ 
for each $\alpha \in \Delta \setminus \{\alpha_{q_1},\ldots,\alpha_{q_k}\}$. In addition, 
one has $B(\alpha,x)=0$ for each $\alpha \in \{\alpha_{q_1},\ldots,\alpha_{q_k}\}$. Thus, we have  
\begin{align*}
B(x,x)=\sum_{1 \leq i,j \leq n} x_ix_jB(\alpha_i,\alpha_j)
=\sum_{i \in I}x_i\sum_{j=1}^n x_jB(\alpha_i,\alpha_j)<0, 
\end{align*}
where $I=\{1,\ldots,n\} \setminus \{q_1, \ldots, q_k\}$, a contradiction. 

Hence $x \in \text{int}(V_\alpha)$ for some $\alpha \in \Delta$. 
Since $B(x,\alpha)>0$ for each $x \in \text{int}(V_\alpha)$, we obtain 
$$\min_{x \in \widehat{Q}}\max_{\alpha \in \Delta}\{B(x,\alpha)\} > 0.$$
If we set $C'=\min_{x \in \widehat{Q}}\max_{\alpha \in \Delta}\{B(x,\alpha)\}$, 
then the assertion holds, as required. 
\end{proof}

Remark that the constant $C'$ appearing above depends only on $B$.

For $\kappa>0$ and $\alpha \in \Delta$, let $U_\alpha^\kappa = \{v \in \widehat{Q} : B(\alpha,v) > \kappa \}$. 
We fix $C=C'-\epsilon$ for a small $\epsilon > 0$ 
such as $\bigcup_{\alpha \in \Delta}U_\alpha^C$ covers $\widehat{Q}$. 
Note that Lemma \ref{hani} guarantees the existence of $C$. Let $U_\alpha=U_\alpha^C$. 
By Lemma \ref{visibility}, one has $U_\alpha \subset \text{int}(V_\alpha)$. 
Thus one can rephrase Lemma \ref{hani} as follows. 
\begin{cor}\label{shinsyuku}
The family of the regions $\{U_\alpha : \alpha \in \Delta\}$ covers $\widehat{Q}$. 
\end{cor}

Let $T=\min_{\alpha \in \Delta}\frac{1}{1-2C|\alpha|_1}$. Then $T>1$ by Lemma \ref{1/2}. 

\begin{prop}\label{T}
For an arbitrary $x \in U_\alpha$ and $y \in V_\alpha$, we have 
\begin{itemize}
\item[(i)] $|B(s_\alpha \cdot x, s_\alpha \cdot y)| > T |B(x,y)|$; 
\item[(ii)] $|B(s_\alpha \cdot x, y)| \geq T |B(x,y)|$. 
\end{itemize}
\end{prop}
\begin{proof}
(i) By Lemma \ref{1/2} and the definition of $C$, one has $C<B(x,\alpha)<\frac{1}{2|\alpha|_1}$. 
Moreover, since $x \in U_\alpha \subset V_1$ and $y \in V_\alpha \subset V_1$, it follows that
$$\frac{1}{|s_\alpha(x)|_1}=\frac{1}{1-2B(x,\alpha)|\alpha|_1} > \frac{1}{1-2C|\alpha|_1} \geq T \text{ and }
\frac{1}{|s_\alpha(y)|_1}=\frac{1}{1-2B(y,\alpha)|\alpha|_1} \geq 1.$$ 
Thus, we obtain 
$$|B(s_\alpha \cdot x, s_\alpha \cdot y)|
=\frac{|B(s_\alpha(x),s_\alpha(y))|}{||s_\alpha (x)|_1 |s_\alpha(y)|_1|}
> T |B(s_\alpha(x),s_\alpha(y))| = T |B(x,y)|.$$

(ii) We have $B(x,y) \leq 0$ by Proposition \ref{hitoshii}. When $B(x,y)=0$, the assertion is obvious. 
Assume that $B(x,y)<0$. Since $B(x,\alpha)>0$ and $B(y,\alpha) \geq 0$, one has 
$$1 - 2 \frac{B(y,\alpha)}{B(x,y)}B(x,\alpha) \geq 1.$$
Hence 
$$|B(s_\alpha \cdot x, y)| =\left| \frac{1 - 2 \frac{B(y,\alpha)}{B(x,y)}B(x,\alpha)}{1 - 2B(x,\alpha)|\alpha|_1} \right| |B(x,y)| \geq T|B(x,y)|.$$
\end{proof}

\begin{proof}[Proof of Theorem \ref{mein} in the case {\bf (a)}] 
By \cite[Theorem 2.7]{hlr}, we know $E \subset \widehat{Q}$. 
What we must show is another inclusion $\widehat{Q} \subset E$. 

Fix $x \in \widehat{Q}$. For $x$, 
we choose an element $w_{x,m}=s_{\alpha_m} \cdots s_{\alpha_1} \in W$ of length $m$ as follows: 
\begin{itemize}
\item For $m=1$, write $w_{x,1}=s_{\alpha}$ for some $\alpha \in \Delta$ such that $x \in U_\alpha$. 
There is at least one such $\alpha$ by Corollary \ref{shinsyuku}. 
\item When we consider $w_{x,m-1} \cdot x$, there exists $\beta \in \Delta$ such that $w_{x,m-1} \cdot x \in U_\beta$. 
We set $w_{x,m}=s_\beta w_{x,m-1}$. 
\end{itemize}
Note that $w_{x,m}$ is not uniquely determined. Moreover, for each $i$, we have $w_{x,i}\not=1$. 
Indeed, by taking $x_i' \in U_\alpha$ such as $w_{x,j} \cdot x_i' \in U_{\alpha_{j+1}}$ for any $j < i$, 
we have $|B(w_{x,i} \cdot x, w_{x,i} \cdot x_i')| > |B(x,x_i')|$ by the proof of Proposition \ref{hodai2}. 

Now, one has $w_{x,m} \cdot x \in \widehat{Q} \setminus V_{\alpha_m}$. 
On the other hand, there exists $y_m$ such that $y_m \in E \cap (\widehat{Q} \setminus V_{\alpha_m})$. 
In fact, take $y \in E \not= \emptyset$. If $y \not\in V_{\alpha_m}$, then let $y_m=y$. 
If $y \in V_{\alpha_m}$, then let $y_m = s_{\alpha_m} \cdot y$. 
By definition of $w_{x,m}$, we see that $w_{x,m-1} \cdot x \in U_{\alpha_m}$. 
Moreover, $s_{\alpha_m} \cdot y_m \in V_{\alpha_m}$. If we set $y_{m-1} = s_{\alpha_m} \cdot y_m$, 
then $y_{m-1} \in E$. By Proposition \ref{T} (i), we see that 
$$|B(w_{x,m} \cdot x, y_m)| = |B(s_{\alpha_m} \cdot (w_{x,m-1} \cdot x), s_{\alpha_m} \cdot (s_{\alpha_m} \cdot y_m))|
\geq T|B(w_{x,m-1} \cdot x, y_{m-1})|.$$ Let 
\begin{align*}
y_{m-2}=
\begin{cases}
y_{m-1}, \;\;\;\;&\text{if }y_{m-1} \in V_{\alpha_{m-1}}, \\
s_{\alpha_{m-1}} \cdot y_{m-1}, \;\;\;\;&\text{if }y_{m-1} \not\in V_{\alpha_{m-1}}. 
\end{cases}
\end{align*}
By Proposition \ref{T} (ii) if $y_{m-2}=y_{m-1}$ and 
Proposition \ref{T} (i) if $y_{m-2}=s_{\alpha_{m-1}} \cdot y_{m-1}$, we obtain that 
$$|B(w_{x,m-1} \cdot x, y_{m-1})| \geq T|B(w_{x,m-2} \cdot x, y_{m-2})|.$$ 
By repeating this estimation, we conclude that $$|B(w_{x,m} \cdot x, y_m)| \geq T^m|B(x, y_0)|.$$
Let $M=\max_{u,v \in \widehat{Q}} |B(u,v)|$. Then $M$ is finite by the compactness of $\widehat{Q}$. 
Moreover, we have $M>0$ and $$0 \leq |B(x,y_0)| \leq \frac{M}{T^m}.$$ 
By taking a large $m$, one can find $y_0 \in E$ such that 
$|B(x,y_0)|$ is arbitrarily small. 
Therefore, we obtain that $x \in E$, as desired. 
\end{proof}

\bigskip

\subsection{The case (b)}

We prove the case {\bf (b)} by induction on the rank of Coxeter groups. 
The following remark guarantees that we can use the induction for the proof.

\begin{rem}\label{daiji}{\em 
For an arbitrary $\Delta_I \subset \Delta$, 
let $W_I$ be a parabolic subgroup of $W$ for $\Delta_I$ and let $B_I$ be the bilinear form associated to $W_I$. 
Then the signature of $B_I$ is $(m,0)$ or $(m-1,0)$ or $(m-1,1)$, where $m=|\Delta_I|$. 
In fact, since $B_I$ is a principal submatrix of $B$, 
the eigenvalues of $B_I$ interlace those of $B$ by \cite[Corollary 2.2]{haemers}. 

For $v \in \text{span}(\Delta_I)$, we define $| v |^I_1$ from $B_I$ in the same manner as $B$ in $V$ 
and consider $\text{span}(\Delta_I)_1=\{v \in \text{span}(\Delta_I) : |v|_1^I=1\}$. 
Then there exists a map $\phi_I$ from $\text{span}(\Delta_I)_1$ to $V_1 \cap \text{span}(\Delta_I)$ 
which is just the normalization $\widehat{\cdot}$ restricted to $\text{span}(\Delta_I)_1$. 
This is actually a homeomorphism. 
Via the map $\phi_I$, we can identify the sets $E \cap \text{span}(\Delta_I)$ 
and $\widehat{Q} \cap \text{span}(\Delta_I)$ with the corresponding sets defined by using $|\cdot|_1^I$. 
}\end{rem}

We divide the case {\bf (b)} into the following two cases: 
\begin{itemize}
\item[{\bf(b-1)}] $\widehat{Q} \not\subset \text{conv}(\widehat{\Delta})$; 
\item[{\bf(b-2)}] $\widehat{Q} \subset \text{conv}(\widehat{\Delta})$ and 
$\widehat{Q} \cap \partial \text{conv}(\widehat{\Delta}) \not=\emptyset$. 
\end{itemize}

For $\alpha \in \Delta$, 
let $\Delta_\alpha=\Delta \setminus \{\alpha\}$, $S_\alpha= S \setminus \{s_\alpha\}$ 
and let $W_\alpha$ denote the parabolic subgroup of $W$ generated by $S_\alpha$. 
When $\alpha=\alpha_j$, we denote $\Delta_j$, $S_j$ and $W_j$ instead of 
$\Delta_{\alpha_j}$, $S_{\alpha_j}$ and $W_{\alpha_j}$, respectively.

Now we see that $\widehat{Q} \cap \text{conv}(\widehat{\Delta}) \not= \emptyset$. 
In fact, we have $B(o,o)<0$ and $B(\widehat{\alpha},\widehat{\alpha})>0$ for all $\alpha \in \Delta$. 
Moreover, since $\widehat{Q}$ is connected, our assumption $\widehat{Q} \not\subset \text{conv}(\widehat{\Delta})$ 
implies that $\widehat{Q} \cap \partial \text{conv}(\widehat{\Delta}) \not= \emptyset$. 
Since $\partial \text{conv}(\widehat{\Delta}) = \bigcup_{j=1}^n \text{conv}(\widehat{\Delta_j})$, 
one has $\widehat{Q} \cap \text{conv}(\widehat{\Delta_j}) \not= \emptyset$ for some $j$'s. 

For $j=1,\ldots,n$, let $A_j=\text{conv}(\widehat{\Delta_j})$ and $H_j=\text{span}(\widehat{\Delta_j})$.

\begin{lem}\label{unun}
Assume $\widehat{Q} \not\subset \text{{\em conv}}(\widehat{\Delta})$. 
Then for a component $D$ of $\widehat{Q} \setminus \text{{\em conv}}(\widehat{\Delta})$, we have the following: 
\begin{itemize}
\item[(i)] $\partial(W \cdot D)$ is $W$-invariant, i.e., $W \cdot \partial(W \cdot D) = \partial(W \cdot D)$; 
\item[(ii)] $\partial D= \widehat{Q} \cap \bigcup_{\partial D \cap A_j \not= \emptyset} A_j$. 
\end{itemize}
\end{lem}
\begin{proof}
(i) Let $y \in W \cdot \partial( W \cdot D)$. Then $y=w \cdot z$ for some $w \in W$ and $z \in \partial( W \cdot D)$. 
Since $W$ acts on $\widehat{Q}$ as homeomorphisms, any neighborhood of $y$ can be expressed as an image by $w$ of some neighborhood of $z$. 
Let $O$ be a neighborhood of $y$ in $\widehat{Q}$. 
Since $w^{-1} \cdot y = z$ is contained in $\partial( W \cdot D)$, 
one has $w^{-1} \cdot O \cap W \cdot D \not= \emptyset$ and 
$w^{-1} \cdot O \cap (\widehat{Q} \setminus W \cdot D) \not= \emptyset$. 
Since $W \cdot D$ and $\widehat{Q} \setminus W \cdot D$ are $W$-invariant, 
one has $w\cdot(w^{-1} \cdot O \cap W \cdot D)= O \cap W \cdot D \not= \emptyset$ and 
$w \cdot (w^{-1} \cdot O \cap (\widehat{Q} \setminus W \cdot D)) =O \cap (\widehat{Q} \setminus W \cdot D) \not= \emptyset$. 
Thus $y$ should belong to $\partial(W \cdot D)$. 
Hence $W \cdot \partial( W \cdot D) \subset \partial( W \cdot D)$. 
On the other hand, the reverse inclusion is obvious. 
Thus $\partial (W \cdot D)$ is $W$-invariant. 

(ii) After some reordering of indices $1,\ldots,n$, we assume that $\partial D \cap A_j \not= \emptyset$ if and only if $1 \leq j \leq k$ for some $k$. 
The inclusion $\partial D \subset \widehat{Q} \cap \bigcup_{j=1}^k A_j$ is obvious. 
Moreover, since $\{v \in A_j : B(v,v) \leq 0\}$ is convex, a point in a segment joining two points of $\{v \in A_j : B(v,v) \leq 0\}$ 
is also a point in this set. This implies that $\{v \in A_j : B(v,v) \leq 0\}$ consists of a single component, 
i.e., $A_j$ does not intersect with any component except for $D$. 
Thus, we obtain $\widehat{Q} \cap A_j \subset \partial D$ for each $j=1,\ldots,k$. 
Hence one has $\partial D=\widehat{Q} \cap \bigcup_{j=1}^k A_j$, as required. 
\end{proof}

Let $D_1,\ldots,D_m$ denote the connected components of $\widehat{Q} \setminus \text{conv}(\widehat{\Delta})$. 
Then each $D_i$ is an open set with respect to the relative topology of $\widehat{Q}$. 

\begin{prop}\label{coco1}
Assume $\widehat{Q} \not\subset \text{{\em conv}}(\widehat{\Delta})$. 
Then one has $E=\bigcup_{i=1}^m \partial (W \cdot D_i)=\partial \bigcup_{i=1}^m W \cdot D_i$. 
\end{prop}
\begin{proof}
Fix an open component $D$ of $\widehat{Q} \setminus \text{conv}(\widehat{\Delta})$, i.e., let $D=D_j$ for some $1 \leq j \leq m$. 
By Lemma \ref{unun} (i), we know that $\partial(W \cdot D)$ is closed and $W$-invariant. 
Hence, by Proposition \ref{keyprop}, in order to prove $\partial (W \cdot D)= \bigcup_{i=1}^m \partial (W \cdot D_i)=E$, 
it suffices to show that $\partial (W \cdot D)$ is contained in $E$.

As in the proof of Lemma \ref{unun}, we assume that $\partial D \cap A_j \not= \emptyset$ if and only if $1 \leq j \leq k$ for some $k$. 
Let $\widehat{Q_j}=\{\widehat{v} : v \in H_j, B_j(v,v)=0\} \subset \widehat{Q}$, 
where $B_j$ is the Coxeter matrix associated with $W_j$. (Note that $B_j$ is a principal submatrix of $B$.) 
Then $\widehat{Q} \cap H_j=\widehat{Q_j}$ and $D \cap H_j = \widehat{Q_j} \setminus A_j$ for each $1 \leq j \leq k$. 
(Note that $\widehat{Q_j} \setminus A_j$ might be non-empty. In fact, if, 
for distinct $i,j$, the parabolic subgroup generated by $S \setminus \{s_{\alpha_i},s_{\alpha_j}\}$ is infinite, 
then $D \cap H_i \not= \emptyset$ and $D \cap H_j \not= \emptyset$.) 
By the inductive hypothesis and Remark \ref{daiji}, one has 
\begin{align}\label{newstyle}
E_j=\widehat{Q_j} \setminus (W_j \cdot (\widehat{Q_j} \setminus A_j)) \;\; \Longleftrightarrow \;\; 
E_j \cup (W_j \cdot (\widehat{Q_j} \setminus A_j)) = \widehat{Q_j}, 
\end{align}
where $E_j$ is the accumulation set of normalized roots of $W_j$. 

Let $x \in \partial(W \cdot D) $. Suppose that $x \not\in E$. 
Since $E$ is a closed set, there exists an open neighborhood $U$ of $x$ in $\widehat{Q}$ such that $U \cap E = \emptyset$. 
Moreover, since $\partial ( W \cdot D) \subset W \cdot \partial D$, 
one has $w \cdot x \in \partial D=\widehat{Q} \cap \bigcup_{j=1}^k A_j$ for some $w \in W$. 
Thus $w \cdot x \in \widehat{Q} \cap A_j$ for some $j \in \{1,\ldots,k\}$. 
Since $(w \cdot U) \cap E_j \subset (w \cdot U) \cap E=\emptyset$, we have $(w \cdot U) \cap E_j= \emptyset$. 
Hence, by the inductive hypothesis \eqref{newstyle} and the equality $\widehat{Q} \cap H_j = \widehat{Q_j}$, 
we have $(w \cdot U) \cap A_j \subset W_j \cdot (\widehat{Q_j} \setminus A_j)$, 
i.e., there exists $w' \in W_j$ such that $w' \cdot ((w \cdot U) \cap A_j) \subset \widehat{Q_j} \setminus A_j =D \cap H_j \subset D$. 
Since $w \cdot x \in w \cdot U \cap A_j$ and $W_j \cdot (\widehat{Q} \setminus A_j) \subset W \cdot D$ 
by $D \cap H_j = \widehat{Q_j} \setminus A_j$, we have $w \cdot x \in W \cdot D$ and therefore $x \in W \cdot D$, 
but this contradicts the assumption $x \in \partial (W \cdot D)$ since $W \cdot D$ is now open. 
Therefore, $x \in E$, i.e., 
\begin{align}\label{mineyama}
\partial (W \cdot D) \subset E. 
\end{align}

Finally, by Lemma \ref{tigauseibun} below, 
we also have $\bigcup_{i=1}^m \partial(W \cdot D_i)=\partial \bigcup_{i=1}^m W \cdot D_i$. 
\end{proof}

\begin{lem}\label{tigauseibun}
	Let $D$ and $D'$ be different components of $\widehat{Q} \setminus \mathrm{conv}(\widehat{\Delta})$.
	Then 
	\[	W \cdot D \cap D' = \emptyset. 	\]
\end{lem}

For the proof of Lemma \ref{tigauseibun}, we prepare two lemmas. 

\begin{lem}\label{hamidashi}
	Let $D$ be a component of $\widehat{Q} \setminus \mathrm{conv}(\widehat{\Delta})$.
	Then a set 
	\[
		K_D = \{x \in D\ :\ B(x,\alpha)<0\ \text{for\ any\ $\alpha \in \Delta$} \}
	\]
	is non-empty.
\end{lem}

Note that for each point $x \in V_1$, we can write $x = \sum_{i=1}^n x_i \widehat{\alpha_i}$. 
Then every point $x = \sum_{i=1}^n x_i \widehat{\alpha_i}$ in a component $D$ of $\widehat{Q} \setminus \mathrm{conv}(\widehat{\Delta})$ 
has at least one index $j$ such that $x_j<0$. Let $J_D \subset \{1,\ldots,n\}$ denote the set of such indices. 
Conversely, if $x_j<0$ for some $j$, then there exists a component $D$ in $\widehat{Q} \setminus \mathrm{conv}(\widehat{\Delta})$. 
By the convexity of $\widehat{Q_-}$, such a component is uniquely determined for $j$. 
Hence, for different components $D$ and $D'$, we have $J_D \cap J_{D'} = \emptyset$.

\begin{proof}[Proof of Lemma \ref{hamidashi}]
	Take $j \in J_D$. For $t \in \mathbb{R}$, let $H(t)$ be an affine subspace in $V_1$ which is parallel to $H_j$ 
        whose $j$-th coordinate is equal to $t$, i.e., 
        $H(t)=\{x \in V_1 : x=\sum_{i=1}^n x_i\widehat{\alpha_i}, \; x_j=t\}$. 
        Note that $H_j = H(0)$ and $\widehat{\alpha_j} \in H(1)$. 
        Since $H(0)$ intersects with $\widehat{Q_-}$, there is $t_0 < 0$ such that $H(t_0)$ is tangent to $\widehat{Q}$. 
	Such a tangent point is unique on $H(t_0)$. Let $x \in H(t_0) \cap \widehat{Q}$ be the tangent point.
	Then each segment connecting $x$ and $\widehat{\alpha}$ ($\alpha \in \Delta$) should intersect with $\widehat{Q_-}$. 
	This implies that $x$ is not visible from $\widehat{\alpha}$. 
        Hence, by Lemma \ref{visibility} (i), we have $B(x,\alpha)<0$ for any $\alpha \in \Delta$. 
	Moreover, since the $j$-th coordinate of $x$ is negative, we have $x \in D$ and thus, $x \in K_D$. 
\end{proof}

\begin{lem}\label{kidounofugou}
	Let $D$ be a component of $\widehat{Q} \setminus \mathrm{conv}(\widehat{\Delta})$.
	Assume that $\ell(ws_\alpha) > \ell(w)$ for $w \in W$ and $\alpha \in \Delta$, 
        where $\ell$ denotes the word length on $(W,S)$. 
	For any $x \in K_D$ and $i \in \{1,\ldots,n\}$, 
	if $(w \cdot x)_i > 0$ then $(ws_\alpha \cdot x)_i > 0$, 
        where $y_i$ denotes the $i$-th coordinate of $y$ when we write $y=\sum_{i=1}^n y_i\widehat{\alpha_i}$ for $y \in V_1$. 
\end{lem}
\begin{proof}
	First, we claim that $|w(x)|_1 > 0$ for any $w \in W$ by induction on $\ell(w)$. 
        Clearly, $|x|_1 =1$ by $x \in V_1$. 
        Now, \cite[Proposition 4.2.5.(i)]{BB05} says that for all $w' \in W$ and $s \in S$, 
        if $\ell(w's) > \ell(w')$ then $w'(\alpha_s)_i > 0$ for any $i$. 
	From this fact together with $B(x,\beta) < 0$ for any $\beta \in \Delta$, 
        we obtain that $|w'(s(x))|_1 = |w'(x)|_1 - 2B(x,\alpha_s)|w'(\alpha_s)|_1 > 0$ by the inductive hypothesis. 
        
	Thus the sign of $(w\cdot x)_i$ is equal to the sign of $w(x)_i$.
	Note that a reflection $s_\alpha \in S$ changes the value of 
	the $i$-th coordinate if and only if $\alpha = \alpha_i$.
	Let $\alpha \neq \alpha_i$. Then the sign of the $i$-th coordinate is preserved.
	Let $\alpha = \alpha_i$. 
	Since $-2B(x,\alpha) > 0$, if $w(x)_i > 0$, then one has $w(s_\alpha(x))_i > 0$ by \cite[Proposition 4.2.5.(i)]{BB05} again. 
\end{proof}

Note that for any $w \in W$, there exists a reduced expression $w=s_1s_2 \cdots s_q$, where $s_i \in S$, 
such that $\ell(s_1 \cdots s_i)>\ell(s_1 \cdots s_{i-1})$ for any $i=1,\ldots,q$. 
It follows from this fact and Lemma \ref{kidounofugou} that for any $x \in K_D$, if $x_j>0$, then $(w \cdot x)_j>0$. 

\begin{proof}[Proof of Lemma \ref{tigauseibun}]
	Suppose that there exists $w \in W$ such that $w \cdot D \cap D' \neq 0$. 
	Let $C$ (resp. $C'$) be the component of $W \cdot D$ (resp. $W \cdot D'$) with $w\cdot D \subset C$ (resp. $D' \subset C'$). 
	Since $C \cap C' \neq \emptyset$, we see that $C = C'$. This shows that $D' \subset W \cdot D$.
	Thus, for $x' \in K_{D'} \subset D'$, 
	there exist $x \in D$ and $w \in W$ so that $x' = w \cdot x$, namely, $x=w^{-1} \cdot x'$. 
	Let $j \in J_{D}$. On the one hand, one has $x_j<0$ by definition of $J_D$. 
        On the other hand, one also has $(w^{-1} \cdot x')_j>0$ by Lemma \ref{kidounofugou} and $j \not\in J_{D'}$, a contradiction.
\end{proof}

\begin{lem}\label{lem:1}
	For a $W$-invariant set $G \subset \widehat{Q} \cap \mathrm{conv}(\widehat{\Delta})$ with non-empty interior, 
	there exists $x \in \mathrm{int}(G)$ such that for any $\alpha \in \Delta$ we have the following:
	\[
		\begin{cases}
		\ x \in V_\alpha &\Longrightarrow x \in \partial V_\alpha,\\
		\ x \notin V_\alpha &\Longrightarrow x_\alpha = 0,
		\end{cases}
	\] 
	where $x_\alpha$ denotes the $\alpha$-th coordinate of $x$. 
\end{lem}

\begin{proof}
	By our assumption, there exist $x \in \overline{G}$ and $y \in \partial G$ such that 
	$|B(x,y)| = \max_{v \in \overline{G}} \min_{u \in \partial G} |B(u, v)| > 0$.
	Then $x$ should belong to $\mathrm{int}(G)$ from Proposition \ref{hitoshii} (a). 
	By Proposition \ref{cover}, there is $\alpha \in \Delta$ such that $x \in V_\alpha$. 
	
	Suppose that $y \notin V_\alpha$. 
	Then $s_\alpha \cdot y \in \partial G \cap \mathrm{int}(V_\alpha)$. 
	Let $z := s_\alpha \cdot y$. 
	Then one has $0\le B(x,\alpha)$, $0< B(z,\alpha )< \frac{1}{2|\alpha|_1}$
	and $B(x, z) < 0$ by Lemma \ref{visibility}, Lemma \ref{1/2} 
	and Proposition \ref{hitoshii}, respectively. 
	Thus $0 < 1-2B(z,\alpha)|\alpha|_1 < 1$. 
	Hence we see that
	\begin{align}
		|B(x,y)| = |B(x, s_\alpha \cdot z)| 
				&= \left| \frac{B(x, z) - 2B(x,\alpha)B(z,\alpha)}{1-2B(z,\alpha)|\alpha|_1} \right| \notag\\
				&= \left| \frac{1-2\frac{B(x,\alpha)}{B(x,z)}B(z,\alpha)}{1-2B(z,\alpha)|\alpha|_1} \right|
						 |B(x,z)| \label{eq:2}\\
				&> |B(x, s_\alpha \cdot y)| \notag.
	\end{align}
	However, this is a contradiction to $|B(x,y)| = \min_{u \in \partial G} |B(x,u)|$. 
	Hence, $y$ should belong to $V_\alpha$. 
	
	Moreover, suppose that $x \notin \partial V_\alpha$. 
	Since $G$ and $\partial G$ are $W$-invariant, one has $s_\alpha \cdot x \in G$ and $s_\alpha \cdot y \in \partial G$.
	Since $|B(x, y)| < |B(s_\alpha \cdot x,s_\alpha \cdot y)|$ by Proposition \ref{hitoshii} (c), 
	from the maximality of $|B(x, y)|$, there is $z' \in \partial G \setminus \{s_\alpha \cdot y\}$ such that 
	$|B(s_\alpha \cdot x, z')| = \min_{u \in \partial G} |B(s_\alpha \cdot x, u)| \le |B(x, y)|$. 
	If $z' \in V_\alpha$, then we obtain that $|B(x,z')| < |B(s_\alpha \cdot x, z')|$ by the similar calculation to \eqref{eq:2}. 
        If $z' \not\in V_\alpha$, then we have $|B(x,s_\alpha \cdot z')| < |B(s_\alpha \cdot x, z')|$ by Proposition \ref{hitoshii} (c). 
        In both cases, we have a contradiction to the choice of $y$, i.e., $|B(x, y)| = \min_{u \in \partial G} |B(x, u)|$. 
	Hence, $x$ should belong to $\partial V_\alpha$.
	
	Therefore, for each $\alpha \in \Delta$, if $x \in V_\alpha$, then $x \in \partial V_\alpha$. 
	This implies that $B(x, \widehat{\alpha}) = 0$ if $x \in V_\alpha$. 
	Moreover, since $x \in \mathrm{conv}(\widehat{\Delta})$, 
	$x$ can be written like $x = \sum_{\delta \in \Delta}x_\delta \widehat{\delta}$,
	with $x_\delta \ge 0$ for each $\delta \in \Delta$. 
	On the other hand, we have 
	\[
		0 = B(x,x) = B\left(x,\sum_{\delta \in \Delta}x_\delta \widehat{\delta}\right) 
				   = \sum_{\delta \in \{\beta\in \Delta : x \notin V_\beta \}} x_\delta B(x, \widehat{\delta}).
	\]
	Since $B(x,\widehat{\delta}) < 0$ for each $\delta \in \Delta$ such that $x \notin V_\delta$ by Lemma \ref{visibility}, 
	we have $x_\delta =0$ for such $\delta$.
\end{proof}

\begin{prop}\label{prop:last}
	For a $W$-invariant subset $K \subset \widehat{Q}$, 
	if $\widehat{Q} \setminus \mathrm{int}(\mathrm{conv}(\widehat{\Delta})) \neq \emptyset$ and 
	$\widehat{Q} \setminus \mathrm{int}(\mathrm{conv}(\widehat{\Delta})) \subset K$, 
        then $\overline{K} = \widehat{Q}$. 
\end{prop}

\begin{proof}
	Let $G = \widehat{Q} \setminus K$.
	If $G$ has no interior then we have the conclusion.
	Thus it suffices to consider the case where $G$ has interior points.
	Note that $G$ is $W$-invariant. 
        
        Since $G \subset \widehat{Q} \cap \mathrm{conv}(\widehat{\Delta})$,
	by Lemma \ref{lem:1}, there exists $x \in \text{int}(G)$ which satisfies that if $x \not\in V_\alpha$  
	the $\alpha$-th coordinate of $x$ equals to $0$, otherwise, $x \in \partial V_\alpha$.
	However, the assumption $\widehat{Q} \setminus \mathrm{int}(\mathrm{conv}(\widehat{\Delta})) \subset K$
	actually implies that $G \subset \mathrm{int}(\mathrm{conv}(\widehat{\Delta}))$. 
	Hence $x_\alpha \neq 0$ for all $\alpha \in \Delta$. 
	Thus $x$ should belong to $\partial V_\alpha$ for all $\alpha \in \Delta$. 
	This contradicts to Remark \ref{atarimae}.
\end{proof}

\begin{proof}[Proof of Theorem \ref{mein} in the case {\bf (b)}]
	Recall that $\widehat{Q} \setminus \mathrm{int}(\mathrm{conv}(\widehat{\Delta})) \neq \emptyset$ in this case.  
	
	\begin{description}
	\item[(b-1)]
	We set $K = \bigcup_{i=1}^m \overline{W \cdot D_i}$.
	Then $K$ obviously contains $\widehat{Q} \setminus \mathrm{int}(\mathrm{conv}(\widehat{\Delta}))$ and is $W$-invariant. 
	Therefore we have $\overline{K}=K= \widehat{Q}$ by applying Proposition \ref{prop:last} to $K$.
	In addition, Proposition \ref{coco1} says that $E = \partial K$. 
	Thus the conclusion follows. 
	
	\item[(b-2)]
	Let $K' = \widehat{Q} \setminus \mathrm{int}(\mathrm{conv}(\widehat{\Delta}))$ and consider $K = W \cdot K'$.
	Then we can also apply Proposition \ref{prop:last} to this $K$ and we get $\overline{K}= \widehat{Q}$.
	
	In this case, for each $x \in K'$, 
        there is $j \in \{1,\ldots,n\}$ such that $x \in \widehat{Q} \cap A_j$. 
	By the inductive hypothesis, $x$ is the accumulation point of normalized roots of $W_j$. Thus $K' \subset E$. 
        Since $E$ is a minimal $W$-invariant subset of $\widehat{Q}$ by Proposition \ref{keyprop}, we have $\overline{K} = E$. 
	Hence we are done. 
\end{description}
\end{proof}

\bigskip

Finally, we conclude this paper with the following remark. 
\begin{rem}\label{kiwotukeru}{\em 
We see that Theorems \ref{mein} (a) and \ref{kidou} imply Conjecture \ref{yosou}. 

\begin{itemize}
\item We first discuss Conjecture \ref{yosou} (i). 
For a Coxeter group $W$ of rank $n$ whose Coxeter matrix is of type $(n-1,1)$, 
as mentioned in Remark \ref{daiji}, 
every bilinear form associated with a parabolic subgroup of rank $m$ is 
of positive type or has the signature $(m-1,1)$.

If $\Delta_I$ is generating, then we can apply Theorem \ref{mein} (a) to $W_I$. 
By using the correspondence induced from $\phi_I$, which is defined in Remark \ref{daiji}, we obtain the conclusion. 
\item For Conjecture \ref{yosou} (ii), 
from the definition of ``generating'' and Theorem \ref{mein} (a), it follows that 
$\overline{F_0}$ is contained in $E$. Moreover, when we take $x \in \widehat{Q} \cap \text{span}(\Delta_I)=E_I \subset E$, 
where $\Delta_I$ is generating, 
it is obvious that $\overline{W \cdot x} \subset \overline{F_0}$. 
Furthermore, by Theorem \ref{kidou}, we know that $E= \overline{W \cdot x}$. Hence, 
$$E = \overline{W \cdot x} \subset \overline{F_0} \subset E.$$ 
Therefore, we conclude that $E=\overline{F_0}$. 
\end{itemize}
}\end{rem}


\end{document}